\documentclass[12pt,english]{article}
\usepackage{babel}
\usepackage[cp1250]{inputenc}
\usepackage[T1]{fontenc}
\usepackage{amsfonts}
\usepackage{amsmath}
\usepackage{amsthm}
\usepackage{hyperref}
\usepackage{enumerate}
\usepackage{graphicx}
\usepackage{pictex}
\usepackage{color}

\newcommand{\eqq}[2]{\begin{equation}  #1  \label{#2} \end{equation}    }

\newcommand{\kr}[1]{{{#1}}}
\newcommand{\al}{\alpha}
\newcommand{\hzj}{{}_{0}H^{1}}
\newcommand{\hzjT}{\hzj(0,T)}
\newcommand{\hza}{{}_{0}H^{\al}}
\newcommand{\hzaT}{\hza(0,T)}

\newcommand*{\spann}{\mathop{\mathrm{span}}}

\newcommand{\hd}{\hspace{0.2cm}}
\newcommand{\no}{\noindent}
\newcommand{\m}[1]{\mbox{#1}}

\newcommand{\OT}{\Omega^{T}}
\newcommand{\N}{\mathbb{N}}

\newcommand{\ppp}{\partial}
\newcommand{\jd}{\frac{1}{2}}

\newtheorem{rem}{{\textbf {Remark}}}
\newtheorem{lem}{{\textbf {Lemma}}}
\newtheorem{prop}{{\textbf {Proposition}}}
\newtheorem{theorem}{\textbf {Theorem}}
\newtheorem{coro}{\textbf  {Corollary} }

\newcommand{\intt}{\int_{0}^{t}}
\newcommand{\intT}{\int_{0}^{T}}

\newcommand{\ia}{I^{\alpha}}
\newcommand{\ija}{I^{1-\alpha}}

\newcommand{\ep}{\varepsilon}
\newcommand{\tamj}{(t-\tau)^{\alpha-1}}
\newcommand{\dt}{d\tau}

\newcommand{\ga}{\Gamma(\alpha)}
\newcommand{\gja}{\Gamma(1-\alpha)}
\newcommand{\jga}{\frac{1}{\ga}}
\newcommand{\jgja}{\frac{1}{\gja}}

\newcommand{\lap}{\Delta}
\newcommand{\la}{\lambda}

\newcommand{\rr}{\mathbb{R}}
\newcommand{\vn}{\varphi_{n}}
\newcommand{\vk}{\varphi_{k}}
\newcommand{\vm}{\varphi_{m}}

\newcommand{\vp}{\varphi}

\newcommand{\da}{D^{\alpha}}
\newcommand{\dda}{D^{2\alpha}}

\newcommand{\ra}{\partial^{\alpha}}
\newcommand{\rb}{\partial^{\beta}}

\newcommand{\un}{u^{n}}

\newcommand{\io}{\int_{\Omega}}
\newcommand{\iop}{\int_{\partial \Omega}}

\newcommand{\izt}{\int_{0}^{t}}
\newcommand{\izth}{\int_{0}^{t-h}}
\newcommand{\iht}{\int_{t-h}^{t}}

\newcommand{\ta}{(t-\tau)^{-\alpha}}

\newcommand{\taj}{(t-\tau)^{-\alpha-1}}

\newcommand{\ld}{L^{2}(\Omega)}

\newcommand{\ldT}{L^{2}(\Omega^{T})}
\newcommand{\hjzo}{H^{1}_{0}(\Omega)}
\newcommand{\hmj}{H^{-1}}
\newcommand{\hmjo}{H^{-1}(\Omega)}

\newcommand{\sij}{\sum_{i,j=1}^{N}}
\newcommand{\sj}{\sum_{j=1}^{N}}
\newcommand{\aij}{a_{i,j}}
\newcommand{\anij}{a^{n}_{i,j}}

\newcommand{\aijxt}{a_{i,j}(x,t)}
\newcommand{\anijxt}{a^{n}_{i,j}(x,t)}
\newcommand{\ankijxt}{a^{n_{k}}_{i,j}(x,t)}

\newcommand{\bjn}{b^{n}_{j}}
\newcommand{\bjnxt}{\bjn(x,t)}
\newcommand{\bjnk}{b^{n_{k}}_{j}}

\newcommand{\sun}{\sum_{k=1}^{n}}
\newcommand{\cnk}{c_{n,k}}
\newcommand{\cnkt}{\cnk(t)}
\newcommand{\cnm}{c_{n,m}}
\newcommand{\cnmt}{\cnm(t)}
\newcommand{\cn}{c_{n}}
\newcommand{\cnt}{c_{n}(t)}
\newcommand{\afn}{\widetilde{A}{}^{n}}

\newcommand{\vkx}{\vk(x)}
\newcommand{\vmx}{\vm(x)}
\newcommand{\fn}{f^{n}}
\newcommand{\fnxt}{\fn (x,t)}
\newcommand{\fjn}{f_{\frac{1}{n}}}

\newcommand{\fjnxt}{\fjn (x,t)}
\newcommand{\fjnk}{f_{1/n_{k}}}
\newcommand{\fjnkxt}{\fjnk (x,t)}

\newcommand{\tma}{t^{1-\alpha}}
\newcommand{\xtj}{X(T_{1})}
\newcommand{\ya}{Y_{\alpha}}

\newcommand{\pon}{\frac{\partial u }{\partial n}}

\newcommand{\sn}[1]{\sum_{#1=1}^{N}}

\newcommand{\ut}{\widetilde{u}}
\newcommand{\unk}{u^{n_{k}}}
\newcommand{\una}[1]{\un (\cdot, #1 )}
\newcommand{\nol}[1]{ \| #1 \|_{L^{2}(\Omega)} }
\newcommand{\nolk}[1]{ \| #1 \|_{L^{2}(\Omega)}^{2} }

\newcommand{\nomk}[1]{ \| #1 \|_{H^{-1}(\Omega)}^{2} }

\newcommand{\ete}{\eta_{\ep}}
\newcommand{\etjn}{\eta_{\frac{1}{n}}}
\newcommand{\etet}{\ete(t+t_{0})}
\newcommand{\unct}{\un (\cdot, t )}
\newcommand{\ct}{(\cdot , t)}

\newcommand{\ddt}{\frac{d}{dt}}

\newcommand*{\esssup}{\mathop{\mathrm{ess\hspace{0.05cm}sup}}\limits}

\newcommand{\sumi}{\sum_{k=0}^{\infty}}

\newcommand{\czb}{\bar{c}_{0}}

\newcommand{\limk}{\lim_{k\rightarrow \infty}}

\newcommand{\hk}{\bar{H}}
\newcommand{\hkk}{\hk^{k}}

\newcommand{\hkks}{(\hk^{k})^{\ast}}

\newcommand{\hkmjs}{(\hk^{2m+1})^{\ast}}
\newcommand{\hkkjs}{(\hk^{2k+1})^{\ast}}

\newcommand{\rlm}{\partial^{m \alpha}}

\newcommand{\inT}{\int_{0}^{T}}

\begin{document}
\title{\bf Initial-boundary value problems for fractional diffusion equations
with time-dependent coeffcients}

\author{ Adam Kubica\footnote{Department of Mathematics and Information
Sciences, Warsaw University of Technology, ul. Koszykowa 75, 00-662 Warsaw,
Poland, E-mail addresses:
A.Kubica@mini.pw.edu.pl},
Masahiro Yamamoto\footnote{Departament of Mathematical Sciences,
The University of Tokyo, Komaba, Meguro, Tokyo - 153, Japan}
\footnote{Corresponding author. E-mail addresses:
myama@ms.u-tokyo.ac.jp (M. Yamamoto)
}
}
\maketitle

\abstract{
We discuss an initial-boundary value problem for a fractional diffusion
equation
with Caputo time-fractional derivative where the coefficients are
dependent on spatial and time variables and the zero Dirichlet boundary
condition is attached.
We prove the unique existence of weak and regular solutions.
}

\vspace{0.2cm}

\no 2010 \textbf{Mathematics Subject Classification}: 35R11, 35K45, 26A33,
34A08.

\section{Introduction}

In this paper we study a parabolic type equation with time fractional Caputo
derivative and general elliptic operator. This problem were considered in many
papers (see \cite{Vasseur}, \cite{Clement}, \cite{Pruss},  \cite{Zacher_hab},
\cite{Zacher_Vol}, \cite{Zacher}), however, in our opinion, it is not
completely understand yet. The main issue which should be explored  more deeply
is the meaning of initial condition $u_{|t=0}$ and the correctness of weak
formulation of the Caputo derivative given for example  in \cite{Zacher}.
In this paper we solved this problem only partially and we will address to it
in another paper. Our results suggest that equations with the Caputo derivative
of order $\alpha \in (0,1)$ requires more regularity of data if $\alpha$ is
equal to or less than $\frac{1}{2}$. Under additional assumptions on data,
we obtain the continuity of solution, but the continuity holds in some
dual space which order depends \m{on $\alpha$.}

The second contribution of our work is the study of general elliptic
operator for which one can not apply Fourier expansion of solution (see
\cite{Yamamoto_1}, \cite{Yamamoto_2}) and it is impossible to reduce the
problem to ordinary fractional equation.

Finally, our approach follows  standard procedure for classical parabolic
problems: first we construct  approximate solution, next we obtain a priori
estimate and further we obtain solution by the weak compactness argument.

Now we recall the definitions of the fractional integration $\ia$  and
the fractional Riemann-Liouville \ derivative
\eqq{\ia f (t) = \jga \izt \tamj  f(\tau) \dt \hd  \m{ for }\alpha>0,}{fI}
\eqq{\partial^{\alpha}f(t) =\ddt \ija f(t)=\jgja \ddt \izt \ta f(\tau)
\dt \m{ for }\alpha\in (0,1).}{fRL}
The formula for $\ia f$ is meaningful for $f\in L^{1}$.
However the formula for the Riemann-Liouville \ derivative  requires more regularity of $f$
and is well defined at least for absolutely continuous $f$ (see
proposition~\ref{propthree} in the appendix) and then $\partial^{\alpha}f $
\m{is in $L^{1}$.} The problem which we shall consider, involves the
fractional Caputo derivative
\eqq{\da f (t) = \partial^{\alpha}[f(\cdot)- f(0)](t)=  \jgja \ddt \izt
\ta [f(\tau)- f(0)]\dt,}{fC}
and this formula is again meaningful for absolutely continuous function $f$.

The aim of this paper is to analyze partial differential equations of
parabolic type which contain the fractional Caputo derivatives.
If we deal with weak solutions, then the Caputo fractional derivative should
be understood in a suitable way. To be more precise we have to formulate the
problem which we analyze in this paper.

Assume that $T<\infty$ and   $\Omega \subset \rr^{N}$ is a  bounded domain
with  smooth boundary, where $N\geq 2$.  We set
$$
\OT = \Omega \times (0,T).
$$
We shall consider the following
problem
\eqq{\left\{
\begin{array}{rllll}
\da u&=& Lu + f & \m{ in } & \Omega^{T}, \\
u_{|\partial \Omega} &=&0 & \m{ for  } & t\in (0,T) \\
u_{|t=0}&=&u_{0} & \m{ in } & \Omega, \\
\end{array}
\right.
}{main}
where
\eqq{Lu(x,t)= \sij  \partial_{i}(\aijxt \partial_{j}u(x,t)) +\sj b_{j}(x,t)
\partial_{j}u(x,t)+ c(x,t)u(x,t),   }{elip}
$\partial_{i}=\frac{\partial}{\partial x_{i}}$ for $i=1, \dots, N$,
and by $\da$ we denote the Caputo fractional time derivative, i.e.
\eqq{
\da u(x,t)=  \jgja \ddt \izt \ta [u(x,\tau)-u(x,0)] \dt.
}{c6}
In the whole paper, the fractional integration and the fractional
differentiation are related only with time variable, and
\eqq{
\ia w(x,t)= \jga \izt \tamj w(x,\tau) \dt, \hd
\ra u(x,t)= \ddt \ija [u(x, \cdot )](t).}{e33}

The definition of the Caputo derivative requires some explanations.
It can be written shortly as $\da u(x,t)= \ddt \ija [u(x,\cdot)-u(x,0)](t)$,
and $u(x, 0)$ is involved.
Therefore we have to guarantee the existence of $u_{|t=0}$ in some sense  and
initial condition (\ref{main})${}_{3}$ should be fulfilled. If these two
demands are satisfied, then for problem (\ref{main}) we could set
\[
\da u(x,t)= \ddt \ija [u(x, \cdot )-u_{0}(x)](t).
\]
The above formula is a starting point in formulating a weak form of the Caputo
derivative related with the problem~(\ref{main}) (we follow \cite{Zacher}).
We shall show that our construction of the solution of (\ref{main}) will fulfil these two demands, at least in the case of $L=\Delta$ (see theorem~\ref{malealpha}). This issue for the general elliptic operator will be examined
in another paper.

We assume that the operator $L$ is uniform elliptic, i.e.,
there exist positive constants $\la$, $\mu$  such that
\eqq{\la |\xi|^{2} \leq \sij \aijxt \xi_{i}\xi_{j} \leq \mu |\xi|^{2}
\hd   \hd \m{ for }  \hd \xi \in \rr^{n}, \hd   t\in [0,T], }{elipt}
with measurable coefficients $\aij$ and $\aij=a_{j,i}$.

We  recall the result by Zacher \cite{Zacher} concerning weak solutions of
(\ref{main}). We introduce notation
\[
W^{{\alpha}}(u_{0}, H^{1}_{0}(\Omega), L^{2}(\Omega))= \{u \in L^{2}(0,T;H^{1}_{0}
(\Omega)): \hd I^{1-\alpha}(u-u_{0})\in {}_{0}H^{1}(0,T;H^{-1}(\Omega)) \},
\]
where the subscript $0$ of ${}_{0}H^{1}$ means vanishing of the
trace for $t=0$ and $AC:= AC[0,T]$ denotes
the space of absolutely continuous functions defined on $[0,T]$
(see definition 1.2, chap. 1 \cite{Samko}).
The following theorem is a special case of theorem~3.1 \cite{Zacher}
(see also corollary~4.1 in \cite{Zacher}).

\begin{theorem}[\cite{Zacher}]
Assume that $\Omega \subseteq \rr^{N}$ is a smooth bounded domain,
$u_{0}\in L^{2}(\Omega)$, $f\in L^{2}(0,T;H^{-1}(\Omega))$,  $b_{j},
c\in L^{\infty}(\OT)$ and (\ref{elipt}) holds. Then there exists a unique
weak solution $u\in W^{{\alpha}}(u_{0}, H^{1}_{0}(\Omega), L^{2}(\Omega))$
of (\ref{main}), i.e.,
\[
\frac{d}{dt} \io I^{1-\alpha}[u(x,t)-u_{0}(x)]\vp (x)dx+ \sij \io \aijxt
\partial_{j}u(x,t)\partial_{i}\vp(x)dx
\]
\eqq{
=\sj \io b_{j}(x,t)\partial_{j}u(x,t)\vp(x)dx + \io c(x,t)u(x,t)\vp(x) dx
+ \langle f(t), \vp\rangle_{H^{-1}\times H^{1}_{0}(\Omega)}
}{a10}
holds for all $\vp\in H^{1}_{0}(\Omega)$ and a.a. $t\in(0,T)$.
Furthermore, the following estimate
\eqq{
\| I^{1-\alpha}(u-u_{0}) \|_{H^{1}(0,T;H^{-1}(\Omega))}
+ \| u \|_{L^{2}(0,T;H^{1}_{0}(\Omega))}
\leq C [\| u_{0} \|_{L^{2}(\Omega)}+ \| f \|_{L^{2}(0,T;H^{-1}(\Omega))}]
}{a11}
holds.
\label{zacher}
\end{theorem}

\begin{rem}
By theorem~\ref{zacher},  for given $u_{0}\in \ld $
$u$ satisfying (\ref{a10}) exists uniquely.
However this result does not guarantee that $u_{|t=0}$ can be defined
adequately.  In particular, it is not clear that $u_{|t=0}=u_{0}$.
In other words, the first term on left-hand side of (\ref{a10})
may not represent the Caputo derivative. However, in the paper \cite{Zacher},
it is remarked (see p.8) that if $\frac{d}{dt}  I^{1-\alpha}[u(x,t)-u_{0}(x)]$
is in $C([0,T];\hmjo)$, then $u \in C([0,T];\hmjo)$ and $u(0)=u_{0}$. In this
paper we develop  this idea in order to overcome the difficulties related to
the definition of the initial value of solution
(see proposition~\ref{initialsecond} in appendix).
\label{remone}
\end{rem}

In the present paper, we first obtain a result similar to
theorem~\ref{zacher}, but its proof is based on special approximating
sequence, which further enables us to improve the regularity of the solutions.

\begin{theorem}
Assume that $\alpha\in (0,1)$, $T>0$, $u_{0}\in L^{2}(\Omega)$ and
$f\in L^{2}(0,T;H^{-1}(\Omega))$. Assume that  (\ref{elipt}) holds
and for some $p_{1},p_{2}\in [2,\frac{2N}{N-2} )$ we have $b\in L^{\infty}(0,T;L^{\frac{2p_{1}}{p_{1}-2}}(\Omega))$, \hd
$c\in  L^{\infty}(0,T;L^{\frac{p_{2}}{p_{2}-2}}(\Omega))$.
Then there exists a unique weak solution $u\in W^{{\alpha}}(u_{0}, H^{1}_{0}(\Omega),
L^{2}(\Omega))$ of (\ref{main}), i.e., (\ref{a10}) holds and
$u$ satisfies  the following  estimate
\[
\| I^{1-\alpha}(u-u_{0}) \|_{H^{1}(0,T;H^{-1}(\Omega))}
+ \| u \|_{L^{2}(0,T;H^{1}_{0}(\Omega))}
+\| u \|_{H^{\frac{\alpha}{2}  }(0,T;\ld ) }
\]
\eqq{ \leq  C \left(\| u_{0} \|_{L^{2}(\Omega)}+ \| f \|_{L^{2}(0,T;H^{-1}(\Omega))}\right), }
{a12}
where $C$ depends only on $\alpha$, $\mu$, $\lambda$, $T$ $\| b \|_
{L^{\infty}(0,T;L^{\frac{2p_{1}}{p_{1}-2}}(\Omega) )}$, $\| c \|_{L^{\infty}
(0,T;L^{\frac{p_{2}}{p_{2}-2}}(\Omega) )}$.

\no Furthermore, if $\alpha >\frac{1}{2}$, then $u\in C([0,T];\hmjo)$  and $u_{|t=0}=u_{0}$.
\label{mainweak}
\end{theorem}

Here and henceforth we set $\frac{2p}{p-2} = \infty$ if $p=2$.

In the case of $L=\Delta$ we  are able to define $u_{|t=0}$ for $\alpha \leq
\frac{1}{2}$.  To formulate the result we need the following notation.

\eqq{\hkk= \left\{w \in H^{k}(\Omega): \hd \Delta^{a}w_{|\partial \Omega}=0
\m{ for } a=0,1,\dots, \left[ \frac{k-1}{2} \right]  \right\},}{e2}

\no and $\hkks$ denotes the dual space to $\hkk$.

\begin{theorem}
Assume that  $u_{0}\in \ld$, $f\in L^{2}(0,T;\hmjo)$ and $u$ is a solution of
(\ref{main}) for $L=\Delta$ given by theorem~\ref{mainweak}. Then
\begin{itemize}
\item
if $\alpha >\frac{1}{2}$, then $\ija [u-u_{0}] \in {}_{0}H^{1}(0,T;\hmjo)$
and $u\in C([0,T];\hmjo)$, \m{$u_{|t=0}=u_{0}$,}
\item
if $\alpha=\frac{1}{2}$ and in addition $\partial^{\frac{1}{2}} f
\in L^{p}(0,T;(\hk^{3})^{\ast})$ for some $p\in (1,2)$, then $u-u_{0}
=I^{1-2\alpha} [u-u_{0}] \in {}_{0}W^{1,p}(0,T;(\hk^{3})^{\ast})$ and
$u\in C([0,T];(\hk^{3})^{\ast})$, \m{$u_{|t=0}=u_{0}$}.
\end{itemize}
If $\alpha\in (0, \frac{1}{2})$ and $k\in \mathbb{N}$ is the smallest number
such that $\frac{1}{2}\leq (k+1)\alpha<1$, then
\begin{itemize}
\item
if $\frac{1}{2} < (k+1)\alpha$ and in addition $\rlm f \in L^{2}(0,T;\hkmjs)$
for $m=1, \dots, k$, then
${{\ddt}} I^{1-(k+1)\alpha} [u-u_{0}] \in L^{2}(0,T;\hkkjs)$ and
$u\in C([0,T];\hkkjs)$, \m{$u_{|t=0}=u_{0}$},

\item
if $\frac{1}{2} = (k+1)\alpha$ and in addition $\rlm f \in L^{2}(0,T;\hkmjs)$
for $m=1, \dots, k$, $\partial^{(k+1)\alpha} f
\in L^{p}(0,T;(\hk^{2k+3})^{\ast})$ for some $p\in (\frac{2}{1+2\alpha},2)$,
then $ I^{1-(k+1)\alpha} [u-u_{0}] \in {{{}_{0}W^{1,p}}} (0,T;\hkkjs)$ and
$u\in C([0,T];\hkkjs)$, \m{$u_{|t=0}=u_{0}$}.

\end{itemize}

\label{malealpha}
\end{theorem}

The above assumption concerning $f$ seems to be essential in any problems
with the Caputo fractional derivative. To illustrate this,
we focus on the case
of $\alpha \in (\frac{1}{4}, \frac{1}{2})$ ($k=1$ in theorem~\ref{malealpha}).
We shall consider simple equation
\eqq{\da w (t) =f(t) \hd \m{ on } \hd [0,T].  }{e3}

We shall show that the assumption $\partial^{\alpha} f \in L^{2}(0,T)$ is
crucial in the problem (\ref{e3}).  For this purpose, we shall
find $f\in L^{2}(0,T)$ such that $\partial^{\alpha}  f \not \in L^{2}(0,T)$,
for which the problem (\ref{e3}) can not have a continuous solution. We recall
that the Caputo fractional derivative $\da w$ makes sense  only if $w(0)$ is
well defined: the alternative definition $\da w (t) = I^{1-\alpha}w'(t)$
requires $w'\in L^{1}(0,T)$, and $w$ should be absolutely continuous
on $[0,T]$.

Suppose the contrary, i.e., there exists a continuous function $w$
such that \linebreak
\m{$\ddt \ija [w-w(0)](t)=f(t)$} holds.
Then applying $I^{1+\alpha}$ to both sides of the equality,
we obtain  $I[w-w(0)](t)=I^{1+\alpha}f(t)$. For $\beta \in (-\frac{1}{2},
- \alpha)$ we set $f(t)=t^{\beta}$. Then $f\in L^{2}(0,T)$, but
$\ra f \not \in L^{2}(0,T)$.
Thus $I[w-w(0)]=c_{\alpha, \beta} t^{1+\alpha+\beta}$ and we see that
$w-w(0)=c_{\alpha, \beta} t^{\alpha+\beta}$. The right-hand side is unbounded
if $t\rightarrow 0^{+}$, and so $w$ can not be continuous. Therefore,
the problem (\ref{e3}) with the Caputo derivative has not a continuous
solution with arbitrary $f\in L^{2}(0,T)$.

Now we formulate the result concerning more regular solution.

\begin{theorem}
Assume that $u_{0}\in \hjzo $, $f\in L^{2}(0,T;\ld )$,   (\ref{elipt}) holds,
\linebreak  $\max_{i,j}\| \nabla \aij \|_{L^{\infty}(\OT)}<\infty$ and
for some $p_{1}\in [2,\frac{2N}{N-2} )$, ${{p_{2}\in[2,4]\cap [2,\frac{2N}{N-2} ) } }$ we have $b\in L^{\infty}
(0,T;L^{\frac{2p_{1}}{p_{1}-2}}(\Omega))$,
\hd $c\in L^{\infty}(0,T;L^{\frac{p_{2}}{p_{2}-2}}(\Omega))$.
Then problem (\ref{main}) has exactly one  solution
$u \in L^{2}(0,T;H^{2}(\Omega))\cap H^{\frac{\alpha}{2} }(0,T;\hjzo)$
such that $\ija [u-u_{0}]\in {}_{0}H^{1}(0,T;\ld)$ and (\ref{main}) holds
almost everywhere in
the sense of (\ref{a10}), where the Caputo derivative $\da u $ is interpreted
as weak time derivative of $\ija [u-u_{0}]$ and the following estimates
\eqq{
\| u \|_{L^{2}(0,T;H^{2}(\Omega))} + \| u \|_{H^{\frac{\alpha}{2} }(0,T;\hjzo)}
 \leq C_{0}(\| u_{0}\|_{\hjzo }+ \| f \|_{L^{2}(0,T;\ld)} ),
}{b4}
\eqq{
\| \ija [u-u_{0}] \|_{H^{1}(0,T;\ld)} \leq  C_{0}(\| u_{0}\|_{\hjzo }
+ \| f \|_{L^{2}(0,T;\ld)} ),
}{b5}
hold, where $C_{0}$ depends only on $\alpha$, $\lambda$, $\mu$, $p_{1}$,
$p_{2}$,  $T$,  $\| \nabla \aij \|_{L^{\infty}(\OT)}$, the Poincar\'e constant
and the $C^{2}$-regularity of $\partial \Omega$ and the norms
$ \| b \|_{L^{\infty}(0,T; L^{\frac{2p_{1}}{p_{1}-2}}(\Omega))}$,
$ \| c \|_{L^{\infty}(0,T; L^{\frac{p_{2}}{p_{2}-2}}(\Omega))}$.

\no Furthermore, if $\alpha >\frac{1}{2}$, then $u\in C([0,T];\ld)$
and $u_{|t=0}=u_{0}$.
\label{mainresult}
\end{theorem}

\section{Notations}
First we introduce the space
\eqq{
Y_{\alpha}(T)=\{h\in C^{1}(0,T]: \hd \tma h'(t)\in C[0,T] \}}{f1}
with the norm $\| h \|_{\ya(T)}= \| h \|_{C[0,T]}+ \| t^{1-\alpha}
h' \|_{C[0,T]}$. Then $Y_{\alpha}(T)$ is a Banach space.  If $H=(h_{1},
\dots , h_{k})$, then we shall write $ H\in \ya(T)$, if $|H|\in \ya(T)$,
where $| \cdot|$ means the maximum norm on $\rr^{k}$.

By assumption (\ref{elipt}) we have $a_{i,j}\in L^{\infty}(\OT)$
(proposition~\ref{elipc}). We denote by $\ete=\ete(t)$ the standard smoothing
kernel, i.e. $\ete \in C^{\infty}_{0}(-\ep/T, \ep/T )$, $\ete$ is nonnegative, $\int_{\rr} \ete (t)dt =1$ \kr{and in addition $\eta_{\ep}(t)=\eta_{\ep}(-t)$}. Then we set
\eqq{\anijxt = \etjn(\cdot) * \aij(x, \cdot )(t),}{b2}
where we extend $\aijxt$ by even reflection for  $t\not \in (0,T)$.  Then
\eqq{\anij \longrightarrow \aij \hd \m{ in } \ldT,}{b3}
and by definition (\ref{b2})  and (\ref{elipt})  we obtain
\eqq{\la |\xi|^{2} \leq \sij \anijxt \xi_{i}\xi_{j} \leq \mu |\xi|^{2}
\hd \hd  \forall t\in [0,T], \hd \forall \xi \in \rr^{n}. }{eliptn}
\no
As a result we have
\eqq{
\anij(g)(t)\equiv \int_{\Omega}\anijxt g(x)dx
\in Y_{\alpha}(T)\quad
\m{ for } n\in \N, \hd i,j\in \{ 1, \dots, N\}, \hd g\in L^{1}(\Omega).}
{d1}

{{\no If we extend function $b_{j}$, $c$ by zero for $t\not \in (0,T)$, then}} the functions $b^{n}_{j}(x,t)$ and $c^{n}(x,t)$ are defined analogously, i.e.
\eqq{b^{n}_{j}(x,t)=\etjn(\cdot ) * b_{j}(x,\cdot)(t), \hd c^{n}(x,t)
=\etjn(\cdot ) * c(x,\cdot)(t), }{c9}
and we have
\eqq{b^{n}_{j}(g)(t)\equiv \int_{\Omega}b^{n}_{j}(x,t) g(x)dx
\in Y_{\alpha}(T), \hd c^{n}(g)(t)\equiv \int_{\Omega}c^{n}(x,t) g(x)dx
\in Y_{\alpha}(T).}{c10}

\section{Approximate solutions}
In this section we shall define a special approximate solution for which
we will be able to obtain appropriate uniform estimates.  We shall assume
that
\eqq{f\in L^{2}(0,T;\hmjo), \hd u_{0}\in \ld, }{c7}
\eqq{b_{j}\in L^{1}(\OT), \hd c\in L^{1}(\OT),}{c8}
and $a_{i,j}$  are measurable and satisfy (\ref{elipt}).

Let $\{ \vn(x)\}_{n\in \mathbb{N}}$ be an  orthonormal basis of $L^{2}
(\Omega)$ such that $- \lap \vn = \la_{n} \vn$  in $\Omega$ and ${\vn}
_{|\partial \Omega}=0$. We will find  approximate solution in the form
\eqq{
\un (x,t)= \sun \cnkt \vkx.
}{m1}
Therefore we have to determine the coefficients $\cnk$. For this purpose we { extend  function $f$ by odd reflection to the interval $(-T,T)$ and we set zero elsewhere. } Then we denote  $f_{\ep}=\eta_{\ep} * f $, where $\eta_{\ep}=\eta_{\ep}(t)$
is a standard smoothing kernel as earlier and we set
\[
\fnxt= \sun \langle f_{\frac{1}{n}}(y,t)\vk(y) \rangle_{H^{-1}\times \hjzo}
\vkx.
\]
We denote
\eqq{L^{n}u(x,t)= \sij  \partial_{i}(\anijxt \partial_{j}u(x,t))
+\sj b^{n}_{j}(x,t)\partial_{j}u(x,t)+c^{n}(x,t)u(x,t),  }{elipn}
where $a_{i,j}^{n}$ are defined in (\ref{b2}) and $b_{j}^{n}$, $c^{n}$
in (\ref{c9}).

In order to determine the coefficients $\cnk$,
we shall consider the following system
\eqq{\left\{
\begin{array}{lll}
\da \un= L^{n}\un + \fn & \m{ in } & \O\\
u_{|t=0}=u_{0}^{n} & \m{ in } & \Omega, \\
\end{array}
\right.
}{appromain}
where $u^{n}_{0}(x)= \sun \io u_{0}(y)\vk(y)dy \vkx$. We define the
coefficients $\cnk$ by a projection of the problem (\ref{appromain})
onto a finite dimensional space span by $\{\vp_{1}, \dots , \vp_{n} \}$.
More precisely,   we multiply (\ref{appromain})${}_{1}$ by $\vm$ and integrate
over $\Omega$.  Then after integrating by parts we have
\[
\da \cnmt =- \sun \sij  \cnkt \io \anijxt \partial_{j}\vkx \partial_{i}\vmx dx
\]
\[
+\sun \sj \cnkt \io \bjnxt \partial_{j}\vkx \vmx dx +\sun \cnkt \io c^{n}(x,t)
\vkx \vmx dx
\]
\eqq{
+ \langle \fjnxt \vmx \rangle_{H^{-1}\times \hjzo},
}{integr1}
where $m=1, \dots, n.$. By  (\ref{eliptn}), (\ref{c7}), (\ref{c8}) and
proposition~\ref{elipc} we deduce that the integrals on the right-hand side
are finite.   We introduce the following notations:
\[
\cnt = (c_{n,1}(t), \dots, c_{n,n}(t)),
\]
\[
A^{n}_{m,k}(t)= \sij  \io \anijxt \partial_{j}\vkx \partial_{i}\vmx dx,
\hd A^{n}(t)=\{A^{n}_{m,k}(t)  \}_{k,m=1}^{n},
\]

\[
B^{n}_{m,k}(t)= \sj  \io b^{n}_{j}(x,t) \partial_{j}\vkx \vmx dx,
\hd B^{n}(t)=\{B^{n}_{m,k}(t)  \}_{k,m=1}^{n},
\]

\[
C^{n}_{m,k}(t)=   \io c^{n}(x,t)\vkx \vmx dx, \hd C^{n}(t)=\{C^{n}_{m,k}(t)
\}_{k,m=1}^{n},
\]

\[
F_{n}(t)= \left( \io f_{\frac{1}{n}}(y,t)\varphi_{1}(y) dy, \dots ,
\io f_{\frac{1}{n}}(y,t)\varphi_{n}(y) dy    \right),
\]
\[
c_{n,0}= \left(  \io u_{0}(y) \varphi_{1}(y)dy, \dots , \io u_{0}(y)
\varphi_{n}(y)dy \right).
\]
Then system (\ref{integr1}) can be written in the following form
\eqq{\left\{
\begin{array}{rll}
\da \cnt &=& - A^{n}(t)\cnt +B^{n}(t)\cnt +C^{n}(t)\cnt + F_{n}(t),\\
\cn(0)&=&c_{n,0}. \\
\end{array}
\right.
}{a1}
We shall show that the above system has an absolutely continuous solution
and then under the assumption $c_{n}\in AC[0,T]$.
By proposition~\ref{propone} the problem (\ref{a1}) is equivalent to the
following integral equation
\eqq{\cnt = c_{n,0}-\ia (A^{n}\cn)(t)+ \ia (B^{n}\cn)(t)+\ia (C^{n}\cn)(t)
+\ia F_{n}(t),}{b1}
where by assumption (\ref{c7}), the function $F_{n}$ is smooth.
By (\ref{d1}) and (\ref{c10}) we have  $A^{n}, B^{n}, C^{n} \in \ya (T)$.
Hence for $\afn\equiv A^{n}-B^{n}-C^{n}$ we also have
\eqq{
\afn \in \ya(T).
}{c11}

\no
Furthermore we define the space
\eqq{X(T)=\{c\in C^{1}((0,T];\rr^n): c(0)=c_{0,n}, \hd \tma c'(t)\in C([0,T];
\rr^{n}) \}.}{e1}
Then, for $c_{1},c_{2}\in X(T)$, defining the distance
$\varrho (c_{1},c_{2})=\|c_{1}-c_{2} \|_{\ya(T)}$, this is a distance
yielding a complete metric on $X(T)$.
We note that $X(T) \subset AC([0,T];\rr^{n})$.

\begin{lem}
For any $n\in \N$ and $T>0$ the system (\ref{b1}) has a  unique solution
in $X(T)$.
\label{approx}
\end{lem}

\begin{proof}
We shall use the Banach fixed point theorem in order to prove the solvability
of (\ref{b1}) in the space (\ref{e1}).
At the first step we shall obtain the solution on some interval $[0,T_{1}]$ and further we shall extend the solution. Hence at the beginning we define   the operator $P$ on $\xtj$ by formula
\eqq{
Pc(t)= c_{n,0}-\ia (\afn c)(t) +\ia F_{n}(t).
}{k1}

Under some smallness assumption on $T_{1}$, we shall obtain the fixed point
of $P$.  Hence we first have to show that $Pc\in \xtj$, provided $c\in \xtj$.
Clearly we have $Pc(0)=c_{n,0}$ and $\afn c$ is continuous and
by proposition~\ref{proptwo} we obtain  the continuity of $Pc$ on $[0,T]$.
From (\ref{c11}) we have $\afn c\in \ya(T_{1})$ and by
propositions~\ref{propthree} and \ref{propfour},
we obtain  $t^{1-\alpha}(Pc)'\in C^{0,\alpha}[0,T_{1}]$, that is,
$Pc\in X(T_{1})$ for arbitrary $T_{1}$. Now we shall show that $P$ is
a contraction on $X(T_{1})$, provided $T_{1}$ is small enough. Indeed, we first  we note that the operator $\ia $ is bounded on $\ya(T_{1})$, and
more precisely from proposition~\ref{propfour} we have
\eqq{
\| \ia h \|_{\ya (T_{1})} \leq C(\alpha)T^{\alpha}_{1}
\| h \|_{\ya (T_{1})}, \hd  \hd h\in \ya (T_{1}). }{g1}

Secondly, we see that
\eqq{\| h_{1} h_{2} \|_{\ya(T_{1})} \leq \| h_{1} \|_{\ya (T_{1})}
\cdot \| h_{2} \|_{\ya (T_{1})}, \hd \hd h_{1}, h_{2}\in \ya (T_{1}). }{h1}
Therefore, if $c_{1},c_{2}\in X(T_{1})$, then form (\ref{g1}) and (\ref{h1})
we have
\[
\| Pc_{1} - Pc_{2} \|_{\ya(T_{1})}= \| \ia (\afn (c_{1} - c_{2}))
\|_{\ya(T_{1})}
\]
\[
\leq C(\alpha)T_{1}^{\alpha} \| \afn (c_{1} - c_{2}) \|_{\ya(T_{1})}
\leq C(\alpha)T_{1}^{\alpha} \| \afn  \|_{\ya(T_{1})} \cdot \| c_{1} - c_{2}
\|_{\ya(T_{1})}.
\]
Hence $P$ is a contraction on $X(T_{1})$, provided
\eqq{
C(\alpha)T_{1}^{\alpha} \| \afn  \|_{\ya(T_{1})} <1,}{i1}
and finally, we obtained a solution of (\ref{b1}) in $X(T_{1})$.

In order to extend the solution,
assume that we have already defined a solution
$\hat{c}$ of (\ref{b1}) on $[0,T_{k}]$, where  $T_{k}>0$.
We shall define  the solution for $t  \in [ T_{k}, T_{k+1}]$ with
$T_{k+1}>T_{k}$. Therefore  we define the set
\[
X_{k}(T_{k+1})=\{c \in C^{1}((0,T_{k+1}];\rr^{n}):\hd c(t)=\hat{c}(t)
\hd \m{ for } t\in [0,T_{k}] \}.
\]
Then $X_{k}(T_{k+1})$ becomes a complete metric space with the metric
$\varrho(c_{1},c_{2})\equiv \| c_{1}-c_{2}\|_{X_{k}(T_{k+1})}
= \| c_{1}'- c_{2}'\|_{C[T_{k},T_{k+1}]}$. Then we define an operator $P$
on $X_{k}(T_{k+1})$ again by formula (\ref{k1}). If $c\in X_{k}(T_{k+1})$,
then by definition of $\hat{c}$,
we have $Pc(t)=\hat{c}(t)$ for $t\in [0,T_{k}]$ and by the same reasoning as
the previous for $X(T_{1})$, we deduce that $Pc\in X_{k}(T_{k+1})$.

Now we shall show that $P$ is a contraction on $X_{k}(T_{k+1})$,
provided $T_{k+1}-T_{k}$ is small enough. Indeed, if $c_{1}, c_{2}
\in X_{k}(T_{k+1})$, then
\[
\| Pc_{1} -Pc_{2} \|_{X_{k}(T_{k+1})}= \| [I^{\alpha}_{T_{k}}
(\afn (c_{1}-c_{2}))]'\|_{C[T_{k},T_{k+1}]},
\]
where $I^{\alpha}_{T_{k}}$ denotes the fractional integration operator
with beginning point $T_{k}$. Using the analog of proposition~\ref{propthree} for $I^{\alpha}_{T_{k}}$ and the equality $c_{1}(T_{k})=c_{2}(T_{k})$,
we obtain
\[
\| Pc_{1} -Pc_{2} \|_{X_{k}(T_{k+1})}=\| I^{\alpha}_{T_{k}}
[(\afn (c_{1}-c_{2}))']\|_{C[T_{k},T_{k+1}]}
\]
\[
\leq C(\alpha) \| \afn \|_{C[T_{k},T_{k+1}]}[T_{k+1}-T_{k}]^{\alpha}
\| c_{1} - c_{2} \|_{X_{k}(T_{k+1})}
\]
\[
+ C(\alpha) \| (\afn)'\|_{C[T_{k},T_{k+1}]}[T_{k+1}-T_{k}]^{\alpha}
\| c_{1} - c_{2} \|_{C[T_{k},T_{k+1}]}.
\]
Using the inequality $\| c_{1} - c_{2} \|_{C[T_{k},T_{k+1}]}
\leq |T_{k+1}-T_{k}| \| c_{1} - c_{2} \|_{X_{k}(T_{k+1})}$,
we deduce that $P$ is a contraction on $X_{k}(T_{k+1})$, provided
\eqq{
C(\alpha)\left[ \| \afn \|_{C[T_{k},T_{k+1}]} + [T_{k+1}-T_{k}] \| (\afn )'
\|_{C[T_{k},T_{k+1}]} \right][T_{k+1}-T_{k}]^{\alpha}<1.
}{l1}
\no
By (\ref{c11}) the quantities $\| \afn  \|_{C[T_{k},T_{n+1}]}$, $\| (\afn )'
\|_{C[T_{k},T_{k+1}]}$ are bounded by $T_{1}^{\alpha-1}\| \afn \|_{\ya(T)}$
and  by iteration we obtain the solution of (\ref{b1}) which belongs to
the space $X(T)$.
The  uniqueness follows from the  uniqueness of the fixed point given by
the Banach theorem.
\end{proof}

\begin{coro}
If $n\in \mathbb{N}$ and   $T>0$, then $\un$ given by (\ref{m1})
and (\ref{b1}) satisfies
\[
\io \da \un (x,t)\vm (x)dx +  \sij   \io \anijxt \partial_{j}\un(x,t)
\partial_{i}\vmx dx
\]
\[
= \sj  \io \bjnxt \partial_{j}\un(x,t) \vmx dx +\io c^{n}(x,t) \un(x,t) \vmx dx
\]
\eqq{
+ \langle \fjnxt, \vmx \rangle_{H^{-1}\times \hjzo},
}{a7}
for $m=1, \dots, n$.  Furthermore, if $x\in \Omega$ and
$\beta\in \mathbb{N}^{N}$, then $\partial^{\beta}_{x} \un (x,\cdot )
\in AC[0,T]$ and $t^{1-\alpha}\partial^{\beta}_{x} \un_{t} \in
C(\overline{\Omega} \times [0,T] )$, provided $\partial \Omega$ is
sufficiently smooth (e.g. $\partial \Omega \in C^{|\beta|+1}$).
\label{coroone}
\end{coro}

\section{Weak solutions}

We shall apply the standard energy method. Briefly speaking, we multiply
the approximate problem (\ref{a7}) by its solution. In order to deal with
the Caputo derivative we need  the following lemma.
\begin{lem}
Assume that $w\in L^{2}(\OT)$ and
\eqq{ w(x,\cdot )\in AC[0,T] \hd  \m{ for } \hd x\in \Omega, }{A3}
and
\eqq{t^{1-\alpha}w_{t}\in L^{\infty}(\OT).}{A4}
Then the following equality
\[
\da \| w(\cdot , t) \|_{L^{2}(\Omega)}^{2} + \frac{\alpha}{\gja}
\izt \taj \io |w(x,t)-w(x,\tau)|^{2}dx \dt
\]
\eqq{+ \jgja t^{-\alpha} \io |w(x,t)-w(x,0)|^{2}dx = 2 \io \da w(x,t) \cdot
w(x,t) dx }{d}
holds.
\label{Ta}
\end{lem}

\begin{proof}
By the definition, we have
\[
2 \io \da w(x,t) \cdot w(x,t) dx - \da \| w(\cdot , t)
\|_{L^{2}(\Omega)}^{2}
\]
\[
= \frac{2}{\gja} \izt \ta \io w_{t}(x, \tau )[w(x,t)- w(x, \tau)] dx \dt
= \frac{2}{\gja}\left( \izth + \iht\right) \equiv I_{1}+I_{2}.
\]
Then
\[
I_{1}= - \jgja \izth \ta \io \left( |w(x,t)-w(x,\tau)|^{2} \right)_{\tau}'
dx \dt
\]
\[
=\frac{\alpha}{\gja} \izth \taj \io |w(x,t)-w(x,\tau)|^{2}dx \dt
+ \jgja t^{-\alpha} \io |w(x,t)-w(x,0)|^{2}dx
\]
\[
- \jgja h^{-\alpha} \io |w(x,t)- w(x,t-h)|^{2}dx.
\]
We denote the last integral by $I_{3}$. Then using assumption (\ref{A4}) we
obtain
\[
I_{3} \leq h^{2-\alpha}|\Omega| \| t^{1-\alpha} w_{t} \|_{L^{\infty}(\OT)}
(t-h)^{2(\alpha-1)}\longrightarrow 0,
\]
if $h \rightarrow 0$. Again using (\ref{A4}) we have the estimate for $I_{2}$
\[
|I_{2}|\leq 2 \| t^{1-\alpha} w_{t} \|_{L^{\infty}(\OT)}
\| w\|_{L^{2}(\OT)}^{2} \iht \ta \tau^{\alpha-1}\dt \longrightarrow 0,
\]
provided $h \rightarrow 0 $. Therefore we obtain (\ref{d})
\end{proof}

\no Now we can prove the first energy estimate for approximate solutions.

\begin{lem}
Assume that $u_{0}\in L^{2}(\Omega)$ and $f\in L^{2}(0,T;H^{-1}(\Omega))$,
$\{ \aij\}_{i,j=1}^{N}$ satisfy (\ref{elipt}), and for some $p_{1},p_{2}\in
[2,\frac{2N}{N-2} )$ we have
$b\in L^{\infty}(0,T;L^{\frac{2p_{1}}{p_{1}-2}}(\Omega))$,
\hd $c\in L^{\infty}(0,T;L^{\frac{p_{2}}{p_{2}-2}}(\Omega))$.
Then for each $t\in [0,T]$ and $n \in \N $ the approximate solution $\un$
satisfies  the following estimate
\[
I^{1-\alpha} \| \una{t} \|_{L^{2}(\Omega)}^{2} +\frac{\alpha}{\gja}
\int_{0}^{t} \int_{0}^{\tau} (\tau - s )^{-\alpha-1}
\nolk{ \una{\tau} - \una{s} }ds \dt
\]
\[
+\frac{1}{\gja} \int^{t}_{0} \tau^{-\alpha} \nolk{ \una{\tau} - \un_{0}(\cdot)}
\dt + \lambda \intt \nolk{ D\una{\tau}}
\]
\eqq{\leq C_{0}\left( \nolk{u_{0}} + \intt \nomk{f(\cdot, \tau) } \dt
+ \delta_{n}\right),   }{a8}
where $C_{0}$ depends only on $\| b \|_{L^{\infty}(0,T;
L^{\frac{2p_{1}}{p_{1}-2}}(\Omega) )}$,
$\| c \|_{L^{\infty}(0,T;L^{\frac{p_{2}}{p_{2}-2}}(\Omega) )}$, $\lambda$,
$\alpha$ and $T$ and $\delta_{n}\rightarrow 0$ uniformly with respect to $t$,
if $n\rightarrow \infty$.
\label{estifirst}
\end{lem}

\begin{proof}
We multiply (\ref{a7}) by $\cnmt$ and sum over $m=1, \dots, n$. Then we have
\[
\io \da \un (x,t)\un (x,t)dx +  \sij   \io \anijxt
\partial_{j}\un(x,t) \partial_{i}\un(x,t) dx
\]
\[
= \sj  \io \bjnxt \partial_{j}\un(x,t) \un (x,t) dx
+\io c^{n}(x,t) |\un(x,t)|^{2}  dx
\]
\[
+ \langle \fjnxt ,\un(x,t) \rangle_{H^{-1}\times \hjzo}.
\]
By corollary~\ref{coroone} the function $\un $ satisfies the assumption of
lemma~\ref{Ta}, so that (\ref{d}) and ellipticity condition (\ref{eliptn})
yield
\[
\jd \da \| \una{t} \|_{L^{2}(\Omega)}^{2} +\frac{\alpha}{2\gja} \int_{0}^{t}
(t-\tau )^{-\alpha-1} \nolk{\una{t} -\una{\tau}  } \dt
\]
\[
+\frac{1}{2\gja}  t^{-\alpha} \nolk{ \una{t} - \un_{0} (\cdot )}
+ \lambda  \nolk{ D\una{t}}
\]
\[
\leq \sj  \io \bjnxt \partial_{j}\un(x,t) \un (x,t) dx
+ \io c^{n}(x,t) |\un(x,t)|^{2}  dx
\]
\eqq{
+  \frac{\lambda}{2}  \nolk{ D\una{t}} + \frac{1}{2\lambda}
\nomk{ \fjn (\cdot, t )}.
}{a9}
First we obtain the estimate for the lower-order terms.
In particular, if we denote $b^{n}=(b^{n}_{1}, \dots , b^{n}_{N})$, then
we have
\[
\da \| \una{t} \|_{L^{2}(\Omega)}^{2} + \lambda  \nolk{ D\una{t}}
\]
\[
\leq  \io |b^{n}(x,t)| |D \un(x,t)| |\un (x,t)| dx +2\io |c^{n}(x,t)|
|\un(x,t)|^{2}  dx + \frac{1}{\lambda} \nomk{ \fjn (\cdot, t )}
\]
\[
\leq \frac{\lambda}{4} \nolk{ D\una{t}}+ \frac{1}{\lambda}
\| b^{n} (\cdot , t)\|_{L^{\frac{2p_{1}}{p_{1}-2} }(\Omega) }^{2} \|
\un (\cdot, t )\|^{2}_{L^{p_{1}}(\Omega)}
\]
\[
+2\| c^{n}(\cdot, t )\|_{L^{\frac{p_{2}}{p_{2}-2} }(\Omega) } \| \un (\cdot, t)
\|^{2}_{L^{p_{2}}(\Omega)} + \frac{1}{\lambda} \nomk{ \fjn (\cdot, t )}
\]

\[
\leq \frac{\lambda}{4} \nolk{ D\una{t}}+ \frac{1}{\lambda} \| b^{n} (\cdot , t)
\|_{L^{\frac{2p_{1}}{p_{1}-2} }(\Omega) }^{2} \left[ \ep_{1}\| D\un (\cdot, t )
\|^{2}_{L^{2}(\Omega)} +c(\ep_{1}, p_{1}) \| \un (\cdot, t  )\|
_{L^{2}(\Omega)}^{2} \right]
\]
\[
+ \| c^{n}(\cdot, t )\|_{L^{\frac{p_{2}}{p_{2}-2} }(\Omega) }
\left[ \ep_{2}\| D\un (\cdot, t )\|^{2}_{L^{2}(\Omega)} +c(\ep_{2}, p_{2})
\| \un (\cdot, t  )\|_{L^{2}(\Omega)}^{2} \right] + \frac{1}{\lambda}
\nomk{ \fjn (\cdot, t )}.
\]
If we take $\ep_{1}$, $\ep_{2}$ small enough, then
\[
\da \| \una{t} \|_{L^{2}(\Omega)}^{2} + \frac{\lambda}{2}  \nolk{ D\una{t}}
\]
\eqq{
\leq
h_{n}(t)\| \un(\cdot, t ) \|^{2}_{L^{2}(\Omega)}+ \frac{1}{\lambda}
\nomk{ \fjn (\cdot, t )},
}{d4}
where the function  $h_{n}(t)$ depends continuously on some powers
of  $\| b^{n} (\cdot , t)\|_{L^{\frac{2p_{1}}{p_{1}-2} }(\Omega) }$,
\hd $\| c^{n} (\cdot , t)\|_{L^{\frac{p_{2}}{p_{2}-2} }(\Omega) }$ and
$\lambda$. If we apply $\ia$ to the sides of (\ref{d4}), then
\eqq{
\| \una{t} \|_{L^{2}(\Omega)}^{2} \leq \| \una{0} \|_{L^{2}(\Omega)}^{2}
+ \frac{1}{\lambda}\ia  \nomk{ \fjn (\cdot, t )}
+ g_{n}(t) \ia \| \una{t} \|_{L^{2}(\Omega)}^{2} ,  }{d5}
where the function $g_{n}(t)$ depends continuously on some powers of
$\| b^{n} \|_{L^{\infty}(0,t;L^{\frac{2p_{1}}{p_{1}-2} }(\Omega)) }$,
\hd $\| c^{n} \|_{L^{\infty}(0,t;L^{\frac{p_{2}}{p_{2}-2} }(\Omega)) }$
and $\lambda$. We apply a generalized Gronwall lemma
(proposition~\ref{gronwall} in appendix) to obtain
\[
\| \una{t} \|_{L^{2}(\Omega)}^{2} \leq  \| \una{0} \|_{L^{2}(\Omega)}^{2}
\sumi g^{k}_{n}(t) \frac{t^{\alpha k }}{\Gamma(1+\alpha k )}
\]
\eqq{
+ \frac{1}{\lambda} \sumi g^{k}_{n}(t) \left( I^{\alpha(k+1)}
\nomk{ \fjn (\cdot, t )}   \right)(t). }{d55}
The convergence of the series follows from the d'Alembert criterion and
\m{$\lim_{x\rightarrow \infty} \frac{\Gamma(x+\alpha)}{\Gamma(x)x^{\alpha}}=1$. }

We once again use the inequality (\ref{a9}). We apply the operator $I$
to both sides of (\ref{a9}). Then using the identity $I=I^{1-\alpha}\ia$
(see theorem 2.5 in \cite{Samko}) and
$\un (x, \cdot )\in AC[0,T]$ for each $x\in \Omega$ and applying
proposition~\ref{propone}, we obtain
\[
\ija  \| \una{t} \|_{L^{2}(\Omega)}^{2} +\frac{\alpha}{\gja} \int_{0}^{t}
\int_{0}^{\tau} \frac{\nolk{\una{\tau } -\una{s}  } }{|\tau -s|^{\alpha+1}}
ds \dt
\]
\[
+\frac{1}{\gja}  \izt \tau^{-\alpha} \nolk{ \una{\tau} - \un_{0} (\cdot )}
\dt   + \lambda \izt  \nolk{ D\una{\tau}} \dt
\]
\[
\leq \frac{t^{1-\alpha}}{\Gamma(2-\alpha)}\nolk{u_{0}} + \frac{1}{\lambda}
\izt \nomk{ \fjn (\cdot, \tau )} \dt
\]
\[
+ 2\sj \izt \io \bjn(x,\tau) \partial_{j}\un(x,\tau) \un (x,\tau) dx \dt
+ 2\izt \io c^{n}(x,\tau) |\un(x,\tau)|^{2}  dx \dt.
\]
If we estimate the last two integrals similarly to the previous, then
\[
\ija  \| \una{t} \|_{L^{2}(\Omega)}^{2} +\frac{\alpha}{\gja}
\int_{0}^{t} \int_{0}^{\tau} \frac{\nolk{\una{\tau } -\una{s}  } }
{|\tau -s|^{\alpha+1}} ds \dt
\]
\[
+ \frac{1}{\gja}  \izt \tau^{-\alpha} \nolk{ \una{\tau} - \un_{0} (\cdot )}
\dt   + \lambda \izt  \nolk{ D\una{\tau}} \dt
\]
\eqq{
\leq \frac{2t^{1-\alpha}}{\Gamma(2-\alpha)}\nolk{u_{0}} + \frac{2}{\lambda}
\izt \nomk{ \fjn (\cdot, \tau )} \dt + g_{n}(t) \izt  \| \una{\tau} \|
_{L^{2}(\Omega)}^{2} \dt.
}{e5}
Using (\ref{d55}) we have
\[
g_{n}(t) \izt  \| \una{\tau} \|_{L^{2}(\Omega)}^{2} \dt \leq  \| \una{0} \|
_{L^{2}(\Omega)}^{2}
\sumi g^{k+1}_{n}(t) \frac{t^{\alpha k +1 }}{\Gamma(2+\alpha k )}
\]
\[
+\frac{1}{\lambda} \sumi g^{k+1}_{n}(t) \left( I^{\alpha(k+1)+1}
\nomk{ \fjn (\cdot, t )}   \right)(t).
\]
Using the Mittag-Leffler function $E_{\alpha, \beta}(z)
= \sumi \frac{z^{k}}{\Gamma(\alpha k +\beta)}$, we can write
\[
\sup_{n} \sup_{t\in (0,T)} \sumi g^{k+1}_{n}(t) \frac{t^{\alpha k +1}}
{\Gamma(2+\alpha k )} = \sup_{n}  \sup_{t\in (0,T)} g_{n}(t) t
E_{\alpha,2}(t^{\alpha} g_{n}(t))
\]
\eqq{=\sup_{n}  g_{n}(T) T E_{\alpha,2}(T^{\alpha} g_{n}(T)) \equiv d_{0}
<\infty, }{d6}
where $d_{0}$ depends only on $\| b \|_{L^{\infty}(0,T;
L^{\frac{2p_{1}}{p_{1}-2}}(\Omega) )}$,
$\| c \|_{L^{\infty}(0,T;L^{\frac{p_{2}}{p_{2}-2}}(\Omega) )}$, $\lambda$,
$\alpha$, $T$, and the convergence of the series follows by
$\lim_{x\rightarrow \infty} \frac{\Gamma(x+\alpha)}{\Gamma(x)x^{\alpha}}=1$.
The second sum is estimated as follows
\[
\sumi g^{k+1}_{n}(t) \left( I^{\alpha(k+1)+1} \nomk{ \fjn (\cdot, t )}
\right)(t)
\]
\[
= \sumi \izt \frac{(t-\tau)^{\alpha(k+1)}g^{k+1}_{n}(t) }
{\Gamma(\alpha(k+1)+1)} \nomk{ \fjn (\cdot, \tau )} \dt
\]
\[
\leq \sumi \frac{t^{\alpha(k+1)}g^{k+1}_{n}(t) }{\Gamma(\alpha(k+1)+1)}
\izt \nomk{ \fjn (\cdot, \tau )} \dt \leq  E_{\alpha,1}(t^{\alpha} g_{n}(t))
\izt \nomk{ \fjn (\cdot, \tau )} \dt.
\]
Next we denote
\eqq{\sup_{n} \sup_{t\in (0,T)} E_{\alpha,1}(t^{\alpha} g_{n}(t))
= \sup_{n} E_{\alpha,1}(T^{\alpha} g_{n}(T))\equiv d_{1}<\infty, }{d7}
where $d_{1}$ depends only on
$\| b \|_{L^{\infty}(0,T;L^{\frac{2p_{1}}{p_{1}-2}}(\Omega) )}$,
$\| c \|_{L^{\infty}(0,T;L^{\frac{p_{2}}{p_{2}-2}}(\Omega) )}$, $\lambda$,
$\alpha$ and $T$.
We note that
\[
\intt \nomk{ \fjn (\cdot, \tau )} \dt \leq \intt \nomk{ f (\cdot, \tau )} \dt
+ \int_{t}^{t+\frac{1}{n}}  \nomk{ f (\cdot, \tau )} \dt.
\]
Thus, setting
\[
\delta_{n}= \sup_{t\in (0,T-\frac{1}{n})} \int_{t}^{t+\frac{1}{n}}
\nomk{ f (\cdot, \tau )} \dt,
\]
and using the assumption concerning $f$, we see that
$\delta_{n}\rightarrow 0$ uniformly with respect to $t$
as $n\rightarrow \infty$.  Therefore
\eqq{ g_{n}(t)\izt \| \una{\tau} \|_{L^{2}(\Omega)}^{2} \dt \leq d_{0} \| u_{0}
\|_{L^{2}(\Omega)}^{2} + \frac{d_{1} }{\lambda} \| f \|^{2}_{L^{2}(0,t; \hmjo)}
+ \frac{d_{1}}{\lambda} \delta_{n},}{d8}
for each $t\in [0,T]$. This estimate together with (\ref{e5}) give (\ref{a8}).
\end{proof}

\begin{proof}[Proof of theorem~\ref{mainweak}]
Denote by $\czb$ the right-hand side of (\ref{a12}). Lemma~\ref{estifirst}
yields a bound for $\un$
\eqq{
\| \nabla \un \|_{L^{2}(0,T;\ld)} + \| \un \|_{ \dot{H}^{\frac{\alpha}{2}}
(0,T;\ld) } \leq \bar{C}_0,
}{a13}
where
\[
\| w \|_{ \dot{H}^{\frac{\alpha}{2}}(0,T;\ld) } = \left( \intT \intT
\frac{ \| w(\cdot , \tau) -w(\cdot, s)\|_{\ld}^{2}}{|\tau - s |^{1+\alpha}} ds
\dt  \right)^{\frac{1}{2}}.
\]
Now we estimate the fractional  derivative $\da \un $. If $w\in \hjzo$, then
$w(x)=\sum_{m=1}^{\infty}d_{m}\vm(x) $, where $d_{m}$ are some numbers and
the series converge in $\hjzo$.
We denote $w^{n}(x)=\sum_{m=1}^{n}d_{m}\vm(x) $. Multiplying (\ref{a7})
by $d_{m}$  and summing from $m=1$ to $n$,  we obtain
\[
\io \da \un (x,t)w (x)dx +  \sij
\io \anijxt \partial_{j}\un(x,t) \partial_{i}w^{n}(x) dx
\]
\[
= \sj  \io \bjnxt \partial_{j}\un(x,t) w^{n}(x) dx +\io c^{n}(x,t) \un(x,t)
w^{n}(x) dx
\]
\[
+ \langle \fjnxt, w(x) \rangle_{H^{-1}\times \hjzo}.
\]
Hence, using proposition~\ref{elipc} and the H\"older
inequality,
we have
\[
\left|
\io \da \un (x,t)w (x)dx  \right|\leq \mu \nol{D\un (\cdot, t )} \nol{Dw^{n}}
+ \| b \|_{L^{\frac{2p_{1}}{p_{1}-2}}(\Omega)  } \| \nabla \un \|_{\ld} \|
w^{n} \|_{L^{p_{1}}(\Omega)}
\]
\[
+ \| c \|_{L^{\frac{p_{2}}{p_{2}-2}}(\Omega)  } \| \un \|_{L^{p_{2}}(\Omega)}\|
w^{n} \|_{L^{p_{2}}(\Omega)} + \| \fjn (\cdot , t ) \|_{H^{-1}(\Omega)}
\| w \|_{\hjzo}.
\]
The function $\un$ is absolutely continuous, and so
we have  $\da \un (x,t)= \frac{d}{d t }\ija [\un (x,t)-u^{n}_{0}(x)]$  and
\[
\left\| \frac{d}{d t }\ija [\un (x,t)-u^{n}_{0}(x)] \right\|_{\hmjo}
= \sup_{\| w \|_{\hjzo}=1} \left|   \io \da \un (x,t)w (x)dx   \right|.
\]
Thus, from the above inequality together with (\ref{a13}), the  Sobolev
embedding and the Poincar\'e inequality yield
\eqq{\sup_{n}  \left\| \frac{d}{d t } \ija [\un -u^{n}_{0}] \right
\|_{L^{2}(0,T;\hmjo)} < \infty.}
{a15}
Therefore, the sequence $\ija [\un -u^{n}_{0}]$ is uniformly bounded in
${}_{0}H^{1}(0,T;\hmjo)$.  By estimates (\ref{a13}), (\ref{a15}) and the weak
compactness argument we obtain $u \in L^{2}(0,T;\hjzo)\cap H^{\frac{\alpha}{2}}
(0,T;\ld)$ and   $v \in {}_{0}H^{1}(0,T;\hmjo)$ such that
there exists a subsequence $\unk$ such that
\eqq{
\unk \rightharpoonup u, \hd \hd \nabla\unk \rightharpoonup \nabla u \hd
\m{ in } L^{2}(0,T;\ld),
}{a16}
\eqq{\ija [\unk -\unk_{0}] \rightharpoonup  v \hd \m{ in } H^{1}(0,T;\hmjo ), }
{a17}
\eqq{
\m{if } p_{2}>4, \m{ then } \unk \rightharpoonup u,  \hd \m{ in }
L^{2}(0,T;L^{\frac{p_{2}}{2}}\Omega),
}{a166}
where the last weak limit is a consequence of the interpolation inequality
\[
\| \unk \|_{L^{2}(0,T;L^{\frac{p_{2}}{2}})(\Omega)} \leq C_{0} \|
\nabla \unk \|_{L^{2}(0,T;\ld)}^{\theta} \| \unk \|_{L^{2}(0,T;\ld)}^{1-\theta},
\]
which holds by $\frac{p_{2}}{2}\in (2,\frac{2N}{N-2})$.

\no First we would like to show that $\ddt \ija [u-u_{0}]$ exists in the
weak sense in $L^{2}(0,T;\hmjo)$ and $\ddt \ija [u-u_{0}]=\ddt v$. Indeed,
we take $\phi\in C^{\infty}_{0}(0,T)$ and $\vp \in \hjzo$ and by the weak
convergence we have
\[
\int_{0}^{T}\phi(t) \left\langle \frac{d}{dt}v(\cdot, t), \vp
\right\rangle_{\hmj \times \hjzo} dt = \limk \int_{0}^{T} \phi(t)
\left\langle \ddt \ija [\unk (\cdot, t)-\unk_{0}], \vp \right\rangle _{\hmj
\times \hjzo} dt
\]
\[
= \limk \int_{0}^{T} \phi(t) \io  \ddt \ija [\unk (x, t)-\unk_{0}(x)]  \vp(x)
dx  dt
\]
\[
= \limk \io \int_{0}^{T} \phi(t)   \ddt \ija [\unk (x, t)-\unk_{0}(x)]
dt \vp(x) dx
\]
\[
= -\limk \io \int_{0}^{T} \phi'(t)    \ija [\unk (x, t)-\unk_{0}(x)]
dt \vp(x) dx
\]
\[
= - \io \int_{0}^{T} \phi'(t)    \ija [u (x, t)-u_{0}(x)]  dt \vp(x) dx,
\]
where the last equality is a consequence of the weak continuity of $\ija$
on the $L^{2}$-spaces (theorem~2.6, \cite{Samko}) and in the previous one
we were allowed to integrate by parts, because $\unk(x,\cdot)\in AC[0,T]$
and so $\ija \unk (x, \cdot )\in AC[0,T]$ by proposition~\ref{propthree}
in appendix. Thus we obtain
\[
\int_{0}^{T}\phi(t) \left\langle v(\cdot, t), \vp \right\rangle_{\hmj
\times \hjzo} dt
= -\int_{0}^{T}\phi'(t) \left\langle \ija [u(\cdot, t)-u_{0}(\cdot)],
\vp \right\rangle_{\hmj \times \hjzo} dt,
\]
and so  $\ddt \ija [u-u_{0}]=\ddt v\in L^{2}(0,T;\hmjo)$ in the  weak
sense
and estimate (\ref{a12}) holds.

Now we shall show the identity (\ref{a10}). By the density argument it is
enough to prove it for $w(x)=\sum_{m=1}^{K} d_{m} \vm(x)$, where $d_{m}$
are arbitrary numbers. We multiply (\ref{a7}) by $d_{m}$ and sum from $m=1$
to $K$. Then, for fixed $t_{0}\in (0,T)$, we multiply the sides by $\etet$,
where $\ete$ is a standard mollifier function  and finally we integrate
with respect to $t\in (0,T)$.  Hence
\[
\intT \etet \io \frac{d}{dt} \ija [\unk(x,t) - \unk_{0}(x)]w(x)dx dt
\]
\[
+\sij  \intT   \io \ankijxt \partial_{j}\unk(x,t) \partial_{i}w(x)\etet dx dt
\]
\[
=\sj  \intT   \io b^{n_{k}}_{j}(x,t) \partial_{j}\unk(x,t) w(x)\etet dx dt
\]
\[
+  \intT  \io c^{n_{k}}(x,t) \unk(x,t) w(x)\etet dx dt
\]
\[
+ \intT \io \fjnkxt w(x)\etet dxdt.
\]

We first  take the limits with as $k\rightarrow \infty$ and next
as $\ep \rightarrow 0$.  For $\ep<  T-t_{0}$, integrating  by parts and
using (\ref{a16}) and continuity of $\ija$ on $L^{2}$ (theorem 2.6
in \cite{Samko}),
we obtain
\[
\intT \etet \io \frac{d}{dt} \ija [\unk(x,t) - \unk_{0}(x)]w(x)dx dt
\]
\[
= -\intT \ete'(t +t_{0}) \io  \ija [\unk(x,t) - \unk_{0}(x)]w(x)dx dt
\]
\[
\underset{k\rightarrow \infty}{\longrightarrow} -\intT \ete'(t +t_{0})
\io  \ija [u(x,t) - u_{0}(x)]w(x)dx dt
\]
\[
=\intT \ete(t +t_{0}) \frac{d}{dt}\io  \ija [u(x,t) - u_{0}(x)]w(x)dx dt
\]
\[
\underset{\ep\rightarrow 0}{\longrightarrow} \frac{d}{dt}
\io  \ija [u(x,t_{0}) - u_{0}(x)]w(x)dx \hd \m{ for a.a. } t_{0}\in (0,T).
\]
For the next term we proceed similarly.
The function $\partial_{i}w(x)\etet$ is smooth in $\Omega^{T}$, and
(\ref{b3}) and (\ref{a16}) yield
\[
\intT   \io \ankijxt \partial_{j}\unk(x,t) \partial_{i}w(x)\etet dx dt
\]
\[
\underset{k\rightarrow \infty}{\longrightarrow} \intT
\io \aijxt \partial_{j}u (x,t) \partial_{i}w(x)\etet dx dt
\]
\[
\underset{\ep\rightarrow 0}{\longrightarrow}  \io \aij(x,t_{0})
\partial_{j}u (x,t_{0}) \partial_{i}w(x) dx   \hd \m{ for a.a. } t_{0}
\in (0,T).
\]
The first term on the right-hand side also converges, because from
the assumption we have $\bjnk \longrightarrow b_{j}$ in $L^{2}(\OT)$.

We have to consider the two cases to deal with the next term on the right-hand
side. If $p_{2}\in [2,4]$, then $c\in \ldT$ and  $\cnk \longrightarrow c$ in
$\ldT$.  Thus $\cnk \unk \rightharpoonup cu $ in $\ldT$.
In the case of $p_{2}>4$ we can write
\[
\left| \intT   \io \left( c^{n_{k}}(x,t) \unk(x,t)- c(x,t)u(x,t)\right) w(x)
\etet dx dt \right|
\]
\[
\leq \intT   \io |c^{n_{k}}(x,t)- c(x,t)| |\unk(x,t)| |w(x)|\etet dx dt
\]
\[
+ \left| \intT   \io c(x,t) (\unk(x,t)- u(x,t)) w(x)\etet dx dt \right|.
\]
The first term converges to zero, because $\cnk \longrightarrow c $ in
$L^{2}(0,T;L^{\frac{p_{2}}{p_{2}-2}}(\Omega))$ and $\unk $ is bounded
in $\left( L^{2}(0,T;L^{\frac{p_{2}}{p_{2}-2}}(\Omega)) \right)^{\ast}
= L^{2}(0,T;L^{\frac{p_{2}}{2}}(\Omega))$ by (\ref{a166}).
In terms of (\ref{a166}), we can deal with the second  term.

Finally we obtain
\[
\intT \io \fjnkxt w(x)\etet dxdt =\intT \langle  \fjnk (t), w\rangle_{\hmj
\times \hjzo } \etet dt
\]
\[
\underset{k\rightarrow \infty}{\longrightarrow}  \intT \langle  f (t),
w\rangle_{\hmj \times \hjzo } \etet dt
\]
\[
\underset{\ep\rightarrow 0}{\longrightarrow} \langle  f (t), w\rangle_{\hmj
\times \hjzo }  \hd \m{ for a.a. } t_{0}\in (0,T),
\]
and (\ref{a10}) is proved { for $\vp=\sum_{m=1}^{K}d_{m}\vp_{m}$ and a.a. $t\in (0,T)$. By density argument we deduce (\ref{a10}) for all $\vp \in H^{1}_{0}(\Omega)$ and a.a. $t\in (0,T)$. }

In the case of $\alpha\in (\frac{1}{2},1)$, the continuity of $u$ and
the equality $u_{|t=0}=u_{0}$ immediately follow from
proposition~\ref{initialsecond}.

It remains to prove the uniqueness of solutions in theorem~\ref{mainweak}.
Assume that $u\in W^{{\alpha}}(0, \hjzo, \hmjo )$ satisfies (\ref{a10}) with
$f\equiv 0 $ and $u_{0}\equiv 0$. Then we set $u_{n}(x,\kr{\tau})
= \sum_{k=1}^{n} d_{k}(\kr{\tau}) \vk (x)$, where $d_{k}(\kr{\tau}) = \io u(x,\kr{\tau})\vk(x)dx$.
Setting $\varphi=\varphi_k$ in (\ref{a10}) and multiplying by $d_{k}(t)$
and summing from
$k=1$ to $n$, we have
\[
\izt \left\langle \frac{d}{d\kr{\tau} } I^{1-\alpha}u(\cdot,\tau),u_{n}(\cdot, \tau)
\right\rangle \dt + \sij \izt \io \aij (x, \tau) \partial_{j}u(x,\tau)
\partial_{i} u_{n}(x , \tau)dx \dt
\]
\[
= \sj \izt \io b_{j}(x,\tau)\partial_{j}u(x,\tau)u_{n}(x , \tau)  dx \dt
+ \izt \io c(x,\tau)u(x,\tau)u_{n}(x , \tau) dx \dt .
\]
The convergence $u_{n} \longrightarrow u$ in $L^{2}(0,t;\hjzo)$ yields
\[
\kr{\lim_{n\rightarrow \infty} }  \izt \left\langle \frac{d}{d\kr{\tau}}  I^{1-\alpha}u(\cdot,\tau),\kr{u_{n}}(\cdot, \tau)
\right\rangle \dt + \sij \izt \io \aij (x, \tau)
\partial_{j}u(x,\tau)\partial_{i} u(x , \tau)dx \dt
\]
\[
= \sj \izt \io b_{j}(x,\tau)\partial_{j}u(x,\tau)u(x , \tau) dx \dt
+ \izt \io c(x,\tau)|u(x,\tau)|^{2}dx \dt .
\]
\kr{We have
\eqq{\izt \left\langle \frac{d}{d\tau}  I^{1-\alpha}u(\cdot,\tau),u_{n}(\cdot, \tau)
\right\rangle \dt=\izt \left\langle \frac{d}{d\tau}  I^{1-\alpha}u_{n}(\cdot,\tau),u_{n}(\cdot, \tau)
\right\rangle \dt.}{nowea}
Indeed,
\[
\izt \left\langle \frac{d}{d\tau}  I^{1-\alpha}u(\cdot,\tau),u_{n}(\cdot, \tau)
\right\rangle \dt=\izt \sun d_{k}(\tau) \left\langle \frac{d}{dt}  I^{1-\alpha}u(\cdot,\tau),\vk(\cdot)
\right\rangle \dt
\]
\[
=\izt \sun d_{k}(\tau) \frac{d}{d \tau} \left\langle   I^{1-\alpha}u(\cdot,\tau),\vk(\cdot)
\right\rangle \dt =\izt \sun d_{k}(\tau) \frac{d}{dt} I^{1-\alpha} \io    u(y,\tau)\vk(y)dy
 \dt
\]
\[
=\izt \sun d_{k}(\tau) \frac{d}{d \tau} I^{1-\alpha} d_{k}(\tau) \dt =\izt \sum_{k,m=1}^{n} d_{m}(\tau) \frac{d}{d \tau} I^{1-\alpha} d_{k}(\tau) \io \vk(x) \vm(x)dx \dt
\]
\[
=\izt \sum_{k,m=1}^{n} d_{m}(\tau) \frac{d}{d \tau} I^{1-\alpha} d_{k}(\tau) \left\langle \vk(\cdot), \vm(\cdot) \right\rangle\dt=\izt   \left\langle  \frac{d}{d \tau} I^{1-\alpha} u_{n}(\cdot, \tau), u_{n}(\cdot, \tau) \right\rangle\dt.
\]
Then we can write
\[
\lim_{n\rightarrow \infty} \izt \left\langle \frac{d}{d \tau}  I^{1-\alpha}u(\cdot,\tau),u_{n}(\cdot, \tau)
\right\rangle \dt
= \lim_{n\rightarrow \infty} \izt \left\langle \frac{d}{d\tau}
I^{1-\alpha}u_{n}(\cdot,\tau),u_{n}(\cdot, \tau) \right\rangle \dt
\]
}
\[
= \lim_{n\rightarrow \infty} \io \izt  \frac{d}{dt}  I^{1-\alpha}u_{n}(x,\tau)
\cdot u_{n}(x, \tau)  \dt dx \geq \frac{t^{-\alpha}}{2\gja}
\liminf_{n\rightarrow \infty} \io \izt  |u_{n}(x,\tau)|^{2} \dt dx
\]
\[
\geq \frac{t^{-\alpha}}{2\gja}  \izt \io |u(x,\tau)|^{2}dx \dt
\]
for a.a. $t\in [0,T]$, where we applied corollary~\ref{lowestiYa}.
Using ellipticity condition (\ref{elipt}), we obtain
\[
\frac{t^{-\alpha}}{2\gja}  \izt \|u(\cdot ,\tau)\|_{\ld}^{2}\dt
+ \frac{\lambda}{2} \izt \| Du(\cdot , \tau) \|_{\ld}^{2}\dt \leq C_{0}
\izt \| u(\cdot , \tau) \|_{\ld}^{2}\dt,
\]
where $C_{0}$ is a constant which depends only on the norms of $b$ in
$L^{\infty}(0,T;L^{\frac{2p_{1}}{p_{1}-2}}(\Omega) )$ and $c$ in $L^{\infty}
(0,T;L^{\frac{p_{2}}{p_{2}-2}}(\Omega) )$.
Therefore we deduce that $u\equiv 0 $ on $\Omega \times (0,t)$,
provided $t$ is small enough.  Repeating this argument, we deduce that
$u\equiv 0$ on $\Omega \times (0,T)$.
\end{proof}

\begin{proof}[Proof of theorem~\ref{malealpha}]
In the case of $L=\Delta$,  the equality (\ref{a7}) has the following form
\eqq{\io \da \un(x,t) \vm(x) dx=\io\un (x,t) \lap \vm (x) dx +\io \fjn(x,t)
\vm (x)dx.   }{a777}
For $\alpha>\frac{1}{2}$ the result is contained in theorem~\ref{mainweak}.
Now assume that $\alpha\in (\frac{1}{4},\frac{1}{2} ]$.
Then we can apply the Riemann-Liouville  \ derivative $\partial^{\alpha} $ to both sides
of (\ref{a777}) and by proposition~\ref{propfive}, we obtain
\[
\io \dda \un(x,t) \vm(x) dx=\io \partial^{\alpha}\un (x,t) \lap \vm (x) dx
+ \io \partial^{\alpha}\fjn(x,t)\vm (x)dx
\]
\eqq{=\io \da \un (x,t) \lap \vm (x) dx +\frac{t^{-\alpha}}{\gja}\io \un(x,0)
\lap \vm(x) dx+\io \partial^{\alpha}\fjn(x,t)\vm (x)dx.}{o4}
If $w\in \hk^{3}$, then there exist constants  $d_{m}$ such that
$w(x) = \sum_{m=1}^{\infty}d_{m}\vmx $ in $\hk^{3}$.
Multiplying (\ref{o4}) by $d_{m}$ and summing over $m$, we have
\[
\io \dda \un(x,t) w(x) dx=\io \da \un (x,t) \lap w (x) dx
+ \frac{t^{-\alpha}}{\gja}\io \un(x,0) \lap w(x) dx
\]
\eqq{
+ \io \partial^{\alpha}\fjn(x,t)w (x)dx.}{o5}
Then, because $\lap w_{|\partial \Omega}=0 $, we can write
\[
\| \dda u^{n} (\cdot, t) \|_{(\hk^{3})^{\ast}} = \sup_{ \| w \|_{\hk^{3}}=1}
\left| \io \dda \un(x,t) w(x) dx \right| \leq \| \partial^{\alpha}
\fjn (\cdot, t) \|_{(\hk^{3})^{\ast}}
\]
\[
+ \sup_{ \| w \|_{\hk^{3}}=1} \left( \| \da \un (\cdot, t) \|_{\hmjo} \|
\lap w \|_{\hjzo}+\frac{t^{-\alpha}}{\gja} \| \un (\cdot, 0) \|_{\ld} \|
\lap w \|_{\ld} \right)
\]
\eqq{
\leq  \| \da \un (\cdot, t) \|_{\hmjo} +\frac{t^{-\alpha}}{\gja} \| u_{0} \|
_{\ld} + \| \partial^{\alpha} \fjn (\cdot, t) \|_{(\hk^{3})^{\ast}}.
}{o6}
We have to consider two cases. If $\alpha \in (\frac{1}{4}, \frac{1}{2})$,
then squaring and integrating both sides of (\ref{o6}), we obtain
\[
\| \dda u^{n}  \|_{L^{2}(0,T;(\hk^{3})^{\ast})}\leq \sqrt{3} \| \da
u^{n}  \|_{L^{2}(0,T;\hmjo)}
\]
\eqq{+ \frac{\sqrt{3}}{\gja} \left(  \frac{T^{1-2\alpha}}{1- 2\alpha} \right)
^{\frac{1}{2}} \| u_{0} \|_{\ld}+c(\al)\| \partial^{\alpha}
f  \|_{L^{2}(0,T;(\hk^{3})^{\ast})},
}{o7}
where we applied the inequality
\eqq{\| \partial^{\alpha} \fjn  \|_{L^{2}(0,T;(\hk^{3})^{\ast})}
\leq c(\al)\| \partial^{\alpha} f  \|_{L^{2}(0,T;(\hk^{3})^{\ast})},}{qq1}
\kr{given in proposition~\ref{estiapproxH} in the appendix}. By the assumption concerning $f$ and (\ref{a15}) we have a uniform bound
for the right-hand side and proceeding as in the proof of
theorem~\ref{mainweak}, we obtain that  a weak solution $u$ of (\ref{main})
satisfies $I^{1-2\alpha}[u-u_{0}] \in {}_{0}H^{1}(0,T;(\hk^{3})^{\ast})$.
Hence, by proposition~\ref{initialsecond} we see that
$u\in C([0,T];(\hk^{3})^{\ast})$ and $u_{|t=0}=u_{0}$.

If $\alpha=\frac{1}{2}$, then we take the $p$-th power of both sides
of (\ref{o6}), where $p\in (1,2)$. Then we have  $u-u_{0} \in {}_{0}W^{1,p}
(0,T;(\hk^{3})^{\ast})$, so that $u\in C([0,T];(\hk^{3})^{\ast})$ and
$u_{|t=0}=u_{0}$.

For $\alpha \in (\frac{1}{6}, \frac{1}{4}]$, we proceed similarly.
We apply the Riemann-Liouville \ derivative $\partial^{2\alpha}$ to both sides of
(\ref{a777}) and taking the $\hk^{5}$-norm by duality we obtain
\eqq{
\| D^{3\alpha} u^{n} (\cdot, t) \|_{(\hk^{5})^{\ast}}  \leq  \| \dda
\un (\cdot, t) \|_{(\hk^{3})^{\ast}} +\frac{t^{-2\alpha}}{\Gamma(1-2\alpha)}
\| u_{0} \|_{\ld}  +\| \partial^{2\alpha} \fjn (\cdot, t)
\|_{(\hk^{5})^{\ast}},
}{o8}
where we used the equality $\lap^{k} w_{|\partial \Omega} =0 $ for $k=1,2$.
Next, applying (\ref{o7}) and proceeding as in the previous case,
we prove the claim.

In general, if $\alpha \in (0, \frac{1}{2})$ and $k $ is the smallest number
such that $\frac{1}{2} <\alpha(k+1)  <1$, then applying the Riemann-Liouville \ derivative
$\partial^{m}$ to both  sides of (\ref{a777}), we obtain
\[
\io D^{(m+1)\alpha} \un(x,t) w(x) dx=\io D^{m\alpha}\un (x,t) \lap w (x) dx
+ \frac{t^{-m \alpha}}{\Gamma(1-m\alpha)} \io \un (x,0) \lap w (x)dx
\]
\[
+\io \partial^{m\alpha}\fjn(x,t)w (x)dx,
\]
where $w \in \hk^{2m+1}$ and  $m=1, \dots, k$.  Then we have
\[
\| D^{(m+1)\alpha} \un(\cdot ,t) \|_{(\hk^{2m+1})^{\ast}}
\]
\[
\leq  \|  D^{m\alpha}\un (\cdot ,t)\|_{(\hk^{2m-1})^{\ast}}
+ \frac{t^{-m\alpha}}{\Gamma(1-m\alpha)}\| u_{0} \|_{\ld}
+ \| {{\partial}}^{m\alpha}\fjn(\cdot ,t)\|_{(\hk^{2m+1})^{\ast}} .
\]
Using these inequalities for $m=1, \dots, k$, we have
\[
\| D^{(k+1)\alpha} \un(\cdot ,t) \|_{\hkkjs}
\]
\eqq{\leq \| \da \un (\cdot, t )\|_{(\hk^{1})^{\ast}}
+t^{-k\alpha}\| u_{0} \|_{\ld}
\sum_{m=1}^{k}\frac{t^{(k-m)\alpha}}{\Gamma(1-m\alpha)}+\sum_{m=1}^{k} \|
\kr{\partial^{m\alpha} \fjn} \|_{\hkmjs}.
}{o9}
We recall that $(\hk^{1})^{\ast} = (\hjzo)^{\ast}=\hmjo$. Hence using the
assumption, the estimate (\ref{a15}), the condition $k\alpha <\frac{1}{2}$ \kr{and proposition~\ref{estiapproxH}},
we obtain a uniform bound for $\frac{d}{dt}I^{1-(k+1)\alpha} [\un- \un_{0}] $
in $ L^{2}(0,T;(\hk^{2k+1})^{\ast})$. Hence $\frac{d}{dt}I^{1-(k+1)\alpha} [u- u_{0}] \in  L^{2}(0,T;(\hk^{2k+1})^{\ast})$, and
applying proposition~\ref{initialsecond} we finish the proof in this case.

Finally, if  $\frac{1}{2}=(k+1)\alpha$ and $w\in \hk^{2k+3}$, then applying
the Riemann-Liouville \ derivative $\partial^{(k+1)\alpha}$ to both  sides of (\ref{a777}),
we obtain
\[
\io D^{\alpha+\frac{1}{2}} \un (x,t) w(x) dx = \io D^{\frac{1}{2}} \un (x,t)
\lap w(x) dx + \frac{t^{-\frac{1}{2}}}{\Gamma(\frac{1}{2})} \io \un (x,0)
\lap w(x) dx
\]
\[
+ \io \partial^{\frac{1}{2}} \fjn (x,t) w(x)dx.
\]
Then using (\ref{o9}) and taking the $p$-th power of both sides,
we obtain a uniform bound for $D^{\alpha+\frac{1}{2}} \un $ in
$L^{p}(0,T;(\hk^{2k+3})^{\ast})$, provided $p\in [1,2)$.
In order to apply proposition~\ref{initialsecond} we choose $p$
such that $\frac{1}{p}<\frac{1}{2}+\alpha$ and the proof is completed.
\end{proof}

\section{Regular solutions}

Now we shall  prove the existence of regular solution of problem (\ref{main}).
We start  with the proof of the second energy estimate for approximating
solutions.

\begin{lem}
Assume that $u_{0}\in \hjzo $, $f\in L^{2}(0,T;\ld )$ and  $\max_{i,j}\|
\nabla \aij \|_{L^{\infty}(\OT)}<\infty$ and
for some $p_{1},p_{2}\in [2,\frac{2N}{N-2})$ we have
$b\in L^{\infty}(0,T;L^{\frac{2p_{1}}{p_{1}-2}}(\Omega))$,
\hd $c\in L^{\infty}(0,T;L^{\frac{p_{2}}{p_{2}-2}}(\Omega))$.
Then for each $t\in [0,T]$ and $n \in \N $ the approximate solution $\un$
satisfies the following estimate
\[
I^{1-\alpha} \| \nabla \una{t} \|_{L^{2}(\Omega)}^{2} +\frac{\alpha}{\gja}
\int_{0}^{t} \int_{0}^{\tau} (\tau - s )^{-\alpha-1} \nolk{\nabla  \una{\tau}
- \nabla \una{s} }ds \dt
\]
\[
+ \frac{1}{\gja} \int^{t}_{0} \tau^{-\alpha} \nolk{ \nabla \una{\tau}
- \nabla \un_{0} (\cdot )} \dt + \frac{\lambda}{32}
\intt \nolk{ D^{2}\una{\tau}}
\]
\eqq{\leq \frac{t^{1-\alpha}}{\Gamma(2-\alpha)}  \| \nabla u_{0}\|^{2}_{\ld}
+ 4\lambda^{-1} \intt \nolk{f(\cdot, \tau) } \dt
+ 4\lambda^{-1} \widetilde{\delta}_{n}
+ \overline{C}_{0} \intt \| D \un(\cdot , \tau ) \|_{\ld}^{2} \dt,  }{a88}
where $\widetilde{\delta}_{n}\rightarrow 0$ uniformly with respect to $t$
as $n\rightarrow \infty$ and
$\overline{C}_{0}$ depends only on $\max_{i,j}\| \nabla \aij  \|
_{ L^{\infty}(\Omega^{T})}$, the regularity of $\partial \Omega$, $p_{1}$,
$p_{2}$, $\lambda$ and norms
$ \| b \|_{L^{\infty}(0,T; L^{\frac{2p_{1}}{p_{1}-2}}(\Omega))}$,
$ \| c \|_{L^{\infty}(0,T; L^{\frac{p_{2}}{p_{2}-2}}(\Omega))}$.
\label{estisecond}
\end{lem}

\begin{proof}
We multiply (\ref{a7}) by $\cnmt \lambda_{m}$ and sum over $m=1, \dots, n$.
Then
\[
-\io \da \un (x,t)\lap \un (x,t)dx -  \sij   \io \anijxt \partial_{j}\un(x,t)
\partial_{i}\lap \un(x,t) dx
\]
\[
=- \sj  \io \bjnxt \partial_{j}\un(x,t) \lap \un (x,t) dx -\io c^{n}(x,t)
\un(x,t) \lap \un (x,t) dx
\]
\[
 -\io \fjnxt \lap \un(x,t) dx.
\]
Using the boundary condition, we have $\da \un_{|\partial \Omega}=0 $,
$\lap  \un_{|\partial \Omega}=0 $ and integrating by parts, we obtain
\[
\io \da \nabla \un (x,t)\nabla  \un (x,t)dx +  \sij   \io \partial_{i}
\left( \anijxt \partial_{j}\un(x,t) \right) \lap \un(x,t) dx
\]
\[
=- \sj  \io \bjnxt \partial_{j}\un(x,t) \lap \un (x,t) dx -\io c^{n}(x,t) \un(x,t) \lap \un (x,t) dx
\]
\[
= -\io \fjnxt \lap \un(x,t) dx.
\]
Applying proposition~\ref{boudary} from the appendix and the Young inequality,
we obtain
\[
\io \da \nabla \un (x,t)\nabla  \un (x,t)dx + \frac{\lambda}{16} \nolk{\nabla^{2}
\unct }
\]
\[
\leq C_{0,n} \nolk{\nabla\unct }+ \frac{4}{\lambda} \nolk{ \fjn \ct} +
\frac{4}{\lambda} \| b^{n} \nabla\un \|^{2}_{\ld}
+ \frac{4}{\lambda} \|  c^{n} \un \|^{2}_{\ld},
\]
where $C_{0,n}$ depends only on the regularity of $\partial \Omega$ and
$\kappa^{n}(t)\equiv \max_{i,j}\| \nabla \anij \ct \|_{L^{\infty}(\Omega)}$.

For any $n\in \N$ and $t\in (0,T)$ we have  $\kappa^{n}(t)\leq \max_{i,j}\|
\nabla \aij \|_{L^{\infty}(\OT)}$ and $C_{0,n}$ are uniformly estimated by
some $C_{0}$, which depends only on  $\max_{i,j}\| \nabla \aij  \|
_{ L^{\infty}(\Omega^{T})}$ and the regularity of $\partial \Omega$.

Similarly as in the proof of lemma~\ref{estifirst}, using the Sobolev
embedding, we obtain
\[
\| b^{n} \nabla\un \|^{2}_{\ld}+  \| c^{n} \un \|^{2}_{\ld}
\]
\[
\leq  \left( \| b^{n} (\cdot , t)\|_{L^{\frac{2p_{1}}{p_{1}-2} }(\Omega) }^{2}
+\| c^{n}(\cdot, t )\|^{2}_{L^{\frac{p_{2}}{p_{2}-2} }(\Omega) } \right)
\left[ \ep \| \nabla^{2}\un (\cdot, t )\|^{2}_{L^{2}(\Omega)} +C \| \un (\cdot, t)
\|_{L^{2}(\Omega)}^{2} \right],
\]
where $C$ depends only on $\ep, p_{1}$ and $ p_{2}$.
Taking $\ep$ small enough, we have
\[
\io \da \nabla \un (x,t)\nabla  \un (x,t)dx + \frac{\lambda}{32} \nolk{\nabla^{2}
\unct }
\]
\[
\leq \overline{C}_{0} \nolk{\nabla\unct }+ \frac{4}{\lambda} \nolk{ \fjn \ct},
\]
where $\overline{C}_{0}$ depends only on  $\max_{i,j}\|
\nabla \aij  \|_{ L^{\infty}(\Omega^{T})}$, the regularity of
$\partial \Omega$, $p_{1}$, $p_{2}$, $\lambda$ and norms $ \| b \|_{L^{\infty}
(0,T; L^{\frac{2p_{1}}{p_{1}-2}}(\Omega))}$,
$\| c \|_{L^{\infty}(0,T; L^{\frac{p_{2}}{p_{2}-2}}(\Omega))}$.

By corollary~\ref{coroone}, the function $\nabla \un $ satisfies
the assumption of lemma~\ref{Ta} and from (\ref{d})  we obtain
\[
\jd \da \| \nabla \una{t} \|_{L^{2}(\Omega)}^{2}
+ \frac{\alpha}{2\gja} \int_{0}^{t}  (t-\tau )^{-\alpha-1}
\nolk{\nabla \una{t} -\nabla \una{\tau}  } \dt
\]
\[
+\frac{1}{2\gja}  t^{-\alpha} \nolk{ \nabla \una{t} - \nabla \un_{0}(\cdot )}
+ \frac{\lambda}{32}  \nolk{ \nabla^{2}\una{t}}
\]
\eqq{
\leq  \overline{C}_{0} \nolk{\nabla\unct }+ \frac{4}{\lambda} \nolk{ \fjn \ct}.
}{a99}
We integrate both sides of (\ref{a99}) with respect to $t\in (0,T)$ and
use the identity $I=I^{1-\alpha}\ia$ and $\nabla \un (x, \cdot )\in AC[0,T]$
for each $x\in \Omega$.  We estimate the second term on the right-hand side
as in the proof of lemma~\ref{estifirst} and after applying
propositions~\ref{propone} \kr{ and \ref{estiapprox}} we have (\ref{a88}).
\end{proof}

\begin{proof}[Proof of theorem~\ref{mainresult}]
Under the   assumptions of  theorem  the existence of a weak solution $u$
is guaranteed by theorem~\ref{mainweak}. Therefore we have to obtain
the additional estimates.  By (\ref{a8}), (\ref{a88}) and the weak
compactness argument, we obtain the bound (\ref{b4}).
Reasoning similarly as in the proof of (\ref{a15}), we obtain
\eqq{\left\| \frac{d}{d t } \ija [\un -u^{n}_{0}] \right\|_{L^{2}(0,T;\ld)}
\leq C(\| u_{0}\|_{\hjzo }+ \| f \|_{L^{2}(0,T;\ld)} ).}{a151}
Hence there exist $w\in \ldT$ and a subsequence, denoted again by $\un$,
such that $\frac{d}{d t } \ija [\un -u^{n}_{0}] \rightharpoonup w$ in $\ldT$.
As in the proof of theorem~\ref{mainweak}, we see that
$\frac{d}{d t } \ija [u -u_{0}]=w$, where the time derivative is
understood in the weak sense.

Finally, from proposition~\ref{initialsecond} we obtain the the continuity of
$u$ with the values in $\ld$, provided  $\alpha >\frac{1}{2}$.
\end{proof}

\section{Appendix}

In this section we collect useful propositions for the proofs.
The basic equality  for the fractional integral is $I^{a}I^{b}f=I^{a+b}f$ and
holds for $f\in L^{1}(0,T)$, where   $a,b$ are positive numbers
(see theorem 2.5 in \cite{Samko}). We also have (see equalities (2.4.33) and
(2.4.44) in \cite{Samko})

\begin{prop}
If $f\in AC[0,T]$ and $\alpha\in (0,1]$, then $\ia \da f(t)=f(t)-f(0)$ and
$\da \ia f(t)=f(t)$.
\label{propone}
\end{prop}

By direct calculation we have
\begin{prop}
If $f\in C[0,T]$ and $\alpha\in (0,1)$, then $\ia f \in C[0,T]$.
\label{proptwo}
\end{prop}

\begin{prop}[lemma A.1 in \cite{KRR}]
If $f\in AC[0,T]$ and $\alpha\in (0,1)$, then $\ia f \in AC[0,T]$ and
$(\ia f)'(t)= \ia f'(t)+\frac{t^{\alpha-1}}{\ga}f(0)$.

\label{propthree}
\end{prop}

\begin{prop}[lemma A.4 in \cite{KRR}]
Assume that $\alpha \in (0,1)$, $f\in AC[0,T]$ and $t^{1-\alpha}f'\in
L^{\infty}(0,T)$. Then
\eqq{|t^{1-\alpha}_{2} (\ia f')(t_{2})-t^{1-\alpha}_{1} (\ia f')(t_{1})|
\leq C_{0} \| t^{1-\alpha} f' \|_{L^{\infty}(0,T)} |t_{2}-t_{1}|^{\alpha},}
{dode}
where $C_{0} $ depends only on $\alpha$.
In particular, $t\mapsto t^{1-\alpha} (\ia f')(t) \in C^{0,\alpha}[0,T]$
and $D^{1-\alpha} f\in C^{0,\alpha}(0,T]$.

\no
\label{propfour}
\end{prop}

In the formulation of lemma A.4 in \cite{KRR} it should be $D^{1-\alpha} f\in C^{0,\alpha}(0,T]$ instead of $D^{\alpha} f\in C^{0,1-\alpha}(0,T]$.

\no
\begin{prop}
Assume that $\alpha, \beta \in (0,1)$, $\alpha+\beta\leq 1$,  $f\in AC[0,T]$.
Then, for the Caputo derivative $D^{\beta}$ and the Riemann-Liouville
derivative $\partial^{\beta}$, defined by (\ref{fC}) and (\ref{fRL})
respectively, the equality $ \partial^{\beta} \da  f(t)=D^{\alpha+\beta}f(t)$
holds.
\label{propfive}
\end{prop}

\begin{proof}
We can write
\[
\partial^{\beta} \da  f(t)= \ddt I^{1-\beta} \da f (t) =\ddt I^{1-\beta} \ddt
\ija [f-f(0)].
\]
Applying proposition~\ref{propthree}, we have
\[
\frac{d^{2}}{dt^{2}} I^{1-\beta} \ija [f-f(0)]=\ddt I^{1-\alpha-\beta} [f-f(0)]
= D^{\alpha+\beta}f(t).
\]
\end{proof}

\begin{prop}[theorem~1 in \cite{Gronwall}]
Assume that $\alpha>0$, $T\in (0,\infty]$, $a,w\in L^{1}_{loc}[0,T)$, $a,g,w$
are nonnegative and $g$ is nondecreasing and bounded. If $w$ satisfies
inequality
\eqq{w(t)\leq a(t)+g(t)(\ia w)(t) \m{ for } t\in [0,T),}{d3}
then
\[
w(t)\leq \sumi g^{k}(t)(I^{\alpha k }a)(t) \m{ for } t\in [0,T).
\]
\label{gronwall}
\end{prop}

For convenience of readers, we recall a simple proof from \cite{Gronwall}.
\begin{proof}
If we apply the operator $(g(t)\ia)^{n}$ to both sides of (\ref{d3}) and
using the  properties of $g$ we deduce that
\[
w(t)\leq \sum_{k=0}^{n-1} g^{k}(t)I^{\alpha k }a(t)+g^{n}(t)I^{\alpha n }w(t).
\]
The last term uniformly tends to $0$, when $n \rightarrow \infty$.
\end{proof}

In the reference to  the remark on p. 8 of \cite{Zacher},
we obtain the following result.

\begin{prop}
Assume that $X$ is a normed vector space,
$u \in L^{1}(0,T;X)$, $p\in (1, \infty)$ and $\ddt( \ija [u-u_{0}])\in
L^{p}(0,T;X)$ and $\ija [u-u_{0}](0)=0$. If $\alpha\in (\frac{1}{p},1]$,
then $u \in C([0,T];X)$ and $u(0)=u_{0}$.
\label{initialsecond}
\end{prop}

\begin{proof}
We have $\ija [u-u_{0}](t)= \izt \frac{d}{ds} (\ija [u-u_{0}](s))ds$ and
\[
\| \ija [u-u_{0}](t) \|_{X}\leq \izt  \| \frac{d}{ds} (\ija [u-u_{0}])(s)
\|_{X} ds \leq t^{1- \frac{1}{p}} \| \ddt \ija [u-u_{0}] \|_{L^{p}(0,T;X)}
\]
for all $t$, where we used the fact that $\ija [u-u_{0}]$ is
absolutely continuous.
Thus
\eqq{ t^{\alpha-1}\| \ija [u-u_{0}](t) \|_{X}\leq t^{\alpha-\frac{1}{p}} \|
\ddt (\ija [u-u_{0}]) \|_{L^{p}(0,T;X)}.}{c4}
On the other hand we have
\eqq{
\ia \ddt (\ija [u-u_{0}])(t)=u(t)-u_{0},
}{c5}
because $\ija [u-u_{0}]\in AC$ and applying proposition~\ref{propthree},
we have
\[
\ia \ddt (\ija [u-u_{0}])(t)=\ddt \ia  (\ija [u-u_{0}])(t)=\ddt I[u-u_{0}](t)
=u(t)-u_{0}.
\]
From theorem~3.6 in \cite{Samko}, we see that $\ia$ is continuous
from $L^{p}(0,T)$ to $C^{0,\alpha-\frac{1}{p}}[0,T]$ and the left-hand side
of (\ref{c5}) is H\"older continuous, and so
$u\in C^{0,\alpha-\frac{1}{p}}([0,T];X)$.
Therefore $u(0)$ is well-defined.
Using (\ref{c4}) and setting  $C_{\alpha}=\Gamma(2-\alpha)$, we have
\[
\| u(0) - u_{0} \|_{X}= C_{\alpha} t^{\alpha-1} \| \ija [u(0)-u_{0}] (t)
\|_{X}
\]
\[
\leq C_{\alpha} t^{\alpha-1} \| \ija [u-u_{0}] (t)\|_{X}+ C_{\alpha}
t^{\alpha-1} \| \ija [u-u(0)] (t)\|_{X}
\]
\[
\leq C_{\alpha} t^{\alpha- \frac{1}{p}}  \| \ddt (\ija [u-u_{0}])
\|_{L^{p}(0,T;X)} + \esssup_{\tau \in (0,t)} \| u(\tau)-u(0) \|_{X}
\]
\[
\leq C_{\alpha} t^{\alpha- \frac{1}{p}}  \| \ddt (\ija [u-u_{0}]) \|_{L^{p}(0,T;X)} + C_{\alpha,p} t^{\alpha- \frac{1}{p}},
\]
where we used the H\"older continuity of $u$.
As $t\rightarrow 0$, we have $u(0)=u_{0}$.
\end{proof}

\begin{prop}
Assume that $\aij=a_{j,i}$ and $\{ \aij (x,t) \}_{i,n=1}^{N}$ define
a uniformly elliptic operator of second order, i.e.,
there exist $\lambda, \mu>0$ such that
\eqq{\lambda |\xi|^{2}\leq \sij \aij (x,t)\xi_{i}\xi_{j} \leq \mu |\xi|^{2}
\hd \hd \m{ for } \hd  \xi\in \rr^{N},
\hd x\in \Omega, \hd   t\in [0,T], }{a2}
holds. Then for all $i,j,x,t$ we have $|\aij (x,t)|\leq \mu$.
\label{elipc}
\end{prop}

\begin{proof}
If we take $k\in \{1, \dots,  N\}$ and set  $\xi_{k}=1$ and $\xi_{l}=0$ for
$l\not = k$, then from (\ref{a2}) we see $0< \lambda \leq a_{k,k} \leq \mu$.
If we take $k,l\in \{1, \dots,  N\}$ and set  $\xi_{k}=1$, $\xi_{l}=1$
and $\xi_{m}=0$ for $m\not = k,l$, then we have
$2\lambda\leq 2 a_{k,l}+a_{k,k}+a_{l,l}\leq 2 \mu$. Thus $2a_{k,l}
\leq 2\mu - a_{k,k}-a_{l,l}\leq 2\mu$ and $2a_{k,l}
\geq 2\lambda - a_{k,k}-a_{l,l}\geq 2 \lambda -2\mu\geq -2\mu$.
\end{proof}

\begin{prop}
Assume that $\Omega\subset \rr^{N}$ is a bounded domain with the boundary of
$C^{2}$ class and  there exist $\lambda, \mu>0$ such that
\eqq{\lambda |\xi|^{2}\leq \sij \aij (x,t)\xi_{i}\xi_{j} \leq \mu |\xi|^{2}
\hd \hd \m{ for } \hd  \xi\in \rr^{N}, \hd x\in \Omega, \hd   t\in [0,T],
}{a33}
holds, where $\aij=a_{j,i}$ and $\kappa(t)=\max_{i,j}\| \nabla \aij (\cdot, t)
\|_{L^{\infty}(\Omega)}$. If $u\in H^{3}(\Omega)$ and $u$ and $\lap u $
vanish on $\partial \Omega$, then
\[
\frac{\lambda}{4}\| \nabla^{2} u \|_{L^{2}(\Omega)}^{2}
- C\| \nabla u \|_{L^{2}(\Omega)}^{2} \leq \sij \io \partial_{i}
(\aijxt \partial_{j} u)\lap u dx
\]
where $C$ depends continuously on $\kappa(t)$ and the $C^{2}$-norm
of $\partial \Omega$ \kr{ and $\nabla^{2}u=\{u_{x_{j}x_{i}}\}_{i,j=1}^{N}$}.
\label{boudary}
\end{prop}

\begin{proof}
We shall follow \cite{Ladyz}. Integrating twice by parts, we have
\[
\sij \io \partial_{i} (\aijxt \partial_{j} u)\lap u dx
= -\sij \io  (\aijxt \partial_{j} u)\partial_{i}\lap u dx
\]
\[
=\sij \io  \aijxt (\partial_{j} \nabla u) (\partial_{i}\nabla u) dx
+ \sij \io  (\nabla \aijxt) (\partial_{j} u) (\partial_{i}\nabla  u) dx
\]
\[
-\sij \iop \aijxt (\partial_{j} u) (\partial_{i} \pon ) dS(x)
\equiv A_{1}+A_{2}+A_{3}.
\]
Here we used the boundary condition $\lap u_{|\partial \Omega} =0$.
Using ellipticity condition (\ref{a33}), we obtain
\[
\lambda \| D^{2} u \|_{L^{2}(\Omega)}^{2} \leq \sij \io
\aijxt (\partial_{j} \nabla u )(\partial_{i}\nabla u) dx .
\]
The term $A_{2}$ is  estimated as follows.
\[
\left|\sij \io (\nabla \aijxt) (\partial_{j} u ) (\partial_{i}\nabla  u)
dx\right|
\leq \kappa(t) \io |\nabla u ||\nabla^{2} u |dx
\leq  \frac{\lambda}{2} \| \nabla^{2} u \|_{L^{2}(\Omega)}^{2}
+ \frac{\kappa^{2}(t)}{2\lambda} \| \nabla u \|^{2}_{L^{2}(\Omega)}.
\]
To finish the proof it is sufficient to obtain the inequality for the
term $A_{3}$
\eqq{\left| \sij \iop \aijxt (\partial_{j} u)( \partial_{i} \pon )dS(x)
\right|
\leq \frac{\lambda}{4} \| \nabla^{2} u \|_{L^{2}(\Omega)}^{2}
+ C \| \nabla u \|^{2}_{L^{2}(\Omega)},  }{a3}
where $C$ depends only on $\kappa(t)$ and the $C^{2}$-norm
of $\partial \Omega$.

For this purpose we first write the function under the integral on the
left-hand side of (\ref{a3}) in coordinates related with
boundary point $x^{0}\in \partial \Omega$. More precisely, for fixed $t$
and $x^{0} \in \partial \Omega$, we define an orthogonal transformation
$P=\{ p_{m,l} \}_{m,l=1}^{N}$ such that for $y=P(x-x^{0})$ we have
$(0, \dots, 1)=P(n(x^{0})-x^{0})$, where $n(x^{0})$ is the outer normal
vector at $x^{0}$. By the assumption concerning the boundary we have
$\Omega \ni x^{0} \mapsto P(x^{0})$ is $C^{1}$.
Then $y_{m}=\sn{l}p_{m,l}(x_{l}-x_{l}^{0})$ and since $P^{T}=P^{-1}$,
we have  $x_{l}-x^{0}_{l}=\sn{m} p_{m,l} y_{m}$ and
$\frac{\partial}{\partial x_{l}}
= \sn{m} p_{m,l}\frac{\partial}{\partial y_{m}}$. Let $(y_{1}, \dots, y_{N-1},
\omega(y_{1}, \dots , y_{N-1}))$ be a parametrization of some neighborhood
of $x^{0}\in \partial \Omega$. Then
\eqq{
\omega(0)=0, \hd \frac{\partial\omega}{\partial y_{i}}(0)=0,
\hd \m{ for } i=1, \dots, N-1.
}{a4}

If we denote $\ut(y,t)=u (x^{0}+P^{T}y,t)$, then using
$u_{|\partial \Omega}=0$ we obtain
\[
\ut(y_{1}, \dots , y_{N-1}, \omega(y_{1}, \dots , y_{N-1}),t)=0,
\]
in some neighborhood of $0\in \rr^{N-1}$. If we take $i,j\in \{1, \dots, N-1
\}$ and differentiate the above equality with respect to
$y_{i}$ and next $y_{j}$, then we have
\eqq{
\frac{\ppp \ut}{\ppp y_{i}}
+ \frac{\ppp \ut}{\ppp y_{N}} \frac{\ppp \omega}{\ppp y_{i}}=0,
\hd \hd
\frac{\ppp^2 \ut}{\ppp y_{i} \ppp y_{j}}
+ \frac{\ppp^2\ut}{\ppp y_{i} \ppp y_{N}} \frac{\ppp \omega}{\ppp y_{j}}
+ \frac{\ppp^2 \ut}{\ppp y_{N} \ppp y_{j}}
\frac{\ppp \omega}{\ppp y_{i}}\frac{\ppp \omega}{\ppp y_{j}}
+ \frac{\ppp\ut}{\ppp y_{N}}\frac{\ppp^2\omega}{\ppp y_{i} \ppp y_{j}} =0,
}{a5}
in some neighborhood of $0\in \rr^{N-1}$.
Hence  (\ref{a4}) yields
\eqq{
\frac{\ppp \ut}{\ppp y_{j}}(0,t)=0, \hd
\frac{\ppp^2\ut}{\ppp y_{i} \ppp y_{j}} (0,t)
= -\frac{\ppp \ut}{\ppp y_{N}}(0,t)
\frac{\ppp^2\omega}{\ppp y_{i} \ppp y_{j}}(0) \hd \m{ for }
i,j\in \{1, \dots , N-1 \}.  }{a6}

On the other hand, using the equality $\frac{\partial u}{\partial n}(x^{0},t)
= \frac{\ppp \ut}{\ppp y_{N}}(0,t)$ and (\ref{a6}), we see
\[
\partial_i \frac{\partial u}{\partial n}(x^{0},t)
= \sn{k} p_{k,i} \frac{\ppp^2\ut}{\ppp y_{N} \ppp y_{k}}(0,t),
\hd \hd \partial_ju(x^{0},t)
= \sn{m}p_{m,j} \frac{\ppp\ut}{\ppp y_{m}}(0,t)
= p_{N,j}\frac{\ppp \ut}{\ppp y_{N}}(0,t).
\]
Thus
\[
\sij \aij(x^{0},t) \partial_j u(x^{0},t) \partial_i
\pon (x^{0},t)
=  \sum_{i,j,k=1}^{N}  \aij(x^{0},t) p_{N,j}p_{k,i}
\frac{\ppp \ut}{\ppp y_{N}}(0,t)
\frac{\ppp^2\ut}{\ppp y_{N} \ppp y_{k}}(0,t)
\]
\[
= \sum_{i,j=1}^{N}  \aij(x^{0},t) p_{N,j}p_{N,i}
\frac{\ppp \ut}{\ppp y_{N}}(0,t)
\frac{\ppp^2\ut}{\ppp y_{N}^2}(0,t)
+ \sum_{k=1}^{N-1}\sum_{i,j=1}^{N}  \aij(x^{0},t) p_{N,j}p_{k,i}
\frac{\ppp \ut}{\ppp y_{N}}(0,t)
\frac{\ppp^2 \ut}{\ppp y_{N} \ppp y_{k}}(0,t).
\]
We shall show that the first sum vanishes. Indeed,  the Laplace operator is
invariant under orthogonal change of variables, so that
$\lap_{y}\ut(0,t)=\lap_{x}u(x^{0},t)$. On the other side,
by the boundary condition we have $\lap_{x}u(x^{0},t)=0$ and
then by (\ref{a6}) we obtain $\frac{\ppp^2\ut}{\ppp y_{N}^2}(0,t)=0$.
Thus
\[
\sij \iop \aij(x^{0},t) \partial_j u(x^{0},t) \partial_i
\pon (x^{0},t) dS(x^{0})
\]
\[
= \frac{1}{2}\sum_{k=1}^{N-1}\sum_{i,j=1}^{N}
\iop \aij(x^{0},t) p_{N,j}(x^{0})p_{k,i}(x^{0})
\frac{\ppp}{\ppp y_k}\left( \left|
\pon(x^{0},t) \right|^{2} \right) dS(x^{0}).
\]
The key observation is that the differentiation with respect to $y_{k}$ for
$k\in \{1, \dots, N-1 \}$ is in fact the differentiation in the tangential
direction on $\partial \Omega$ and we can integrate by parts.
Therefore,
\[
\sij \iop \aij(x^{0},t) \partial_ju(x^{0},t) \partial_i
\pon (x^{0},t) dS(x^{0}) = \iop K(x,t)\left|  \pon(x,t) \right|^{2} dS(x),
\]
where $\Omega \ni x\mapsto K(x,t)$ is a continuous function and
$\| K(\cdot, t)\|_{C(\partial \Omega)}$ depends only on $\kappa(t)$ and
the $C^{2}$-regularity of $\partial \Omega$. Hence using inequality (21) in
\cite{LadyH}, we have
\[
\left| \sij \iop \aij(x,t) \partial_j u(x,t) \partial_i
\pon (x,t) dS(x)\right|
\leq \hat{C} \iop \left|  \pon(x,t) \right|^{2} dS(x)
\]
\[
\leq \ep \| \nabla^{2} u \|^{2}_{L^{2}(\Omega)}
+ C(\ep)\| \nabla u \|^{2}_{L^{2}(\Omega)},
\]
where $C(\ep)$ depends only on $\ep$, $\kappa(t)$ and the $C^{2}$ regularity of
$\partial \Omega$. If we take $\ep=\frac{\lambda}{4}$, then we get (\ref{a3})
and the proof is finished.
\end{proof}

The following proposition can be obtained formally by integration of
the equality (17) in  lemma~2.1 \cite{Zacher} with $k(t)=\jgja t^{-\alpha}$.
However, the  function $k(t)$ does not belong to $W^{1,1}(0,T)$ and we can
not apply this lemma directly.
\begin{prop}
If $w \in AC[0,T]$ then for $\alpha \in (0,1)$ the following equality
\[
\int_{0}^{T} \ra w(t) \cdot w(t) dt
= \frac{\alpha}{4} \inT\inT \frac{|w(t)- w(\tau)|^{2}}{|t-\tau|^{1+\alpha}}
\dt dt
\]
\eqq{+\frac{1}{2} \jgja  \inT [(T-t)^{-\alpha}+t^{-\alpha}]|w(t)|^{2}dt }{n10}
holds, where $\ra $ denotes the Riemann-Liouville \ derivative.
In particular, for $t\in (0,T)$ the inequality
\eqq{\int_{0}^{t} \ra w(\tau) \cdot w(\tau) \dt
\geq \frac{t^{-\alpha}}{2\gja} \izt |w(\tau)|^{2} \dt}{n11}
holds.
\label{keyesti}
\end{prop}

\begin{proof}
We first note that the left-hand side of (\ref{n10}) is finite,
because by proposition~\ref{propthree} $\ija w $ is absolutely continuous and
$\ra w=\frac{d}{dt} \ija w $ is in $L^{1}(0,T)$. Then we calculate
\[
\int_{0}^{T} \ra w(t) \cdot w(t) dt = \int_{0}^{T} \ddt \ija [w(t) -w(0) ]
\cdot w(t) dt + \frac{w(0)}{\gja} \inT t^{-\alpha} w(t) dt
\]
\[
= \int_{0}^{T}  \ija w'(t) \cdot w(t) dt + \frac{w(0)}{\gja}
\inT t^{-\alpha} w(t) dt,
\]
where we applied proposition~\ref{propthree}. Next, by definition of $\ia$
we have
\[
\jgja \int_{0}^{T} \izt  \ta  w'(\tau) \dt  \cdot w(t) dt
+ \frac{w(0)}{\gja} \inT t^{-\alpha} w(t) dt
\]
\[
=\jgja \int_{0}^{T} \izt  \ta  w'(\tau)   \cdot [w(t) -w(\tau)]\dt  dt
+ \frac{w(0)}{\gja} \inT t^{-\alpha} w(t) dt
\]
\[
+ \jgja \int_{0}^{T} \izt  \ta  w'(\tau)   w(\tau)\dt  dt
\]
\[
=- \jd \jgja \int_{0}^{T} \izt  \ta  \left( |w(t) -w(\tau)|^{2}  \right)
_{\tau} \dt  dt+ \frac{w(0)}{\gja} \inT t^{-\alpha} w(t) dt
\]
\[
 +\jd \jgja \int_{0}^{T} \izt  \ta  \left( |w(\tau)|^{2}\right)_{\tau} \dt dt
\]
\[
= \frac{\alpha}{2} \jgja \int_{0}^{T} \izt  \taj   |w(t) -w(\tau)|^{2} \dt dt
- \jd \jgja \int_{0}^{T}   \ta  |w(t) -w(\tau)|^{2}  \Big|_{\tau=0}^{\tau =t}
dt
\]
\[
+ \frac{w(0)}{\gja} \inT t^{-\alpha} w(t) dt+\jd \jgja \int_{0}^{T}
\left( |w(\tau)|^{2}  \right)_{\tau} \int_{\tau}^{T}  \ta   dt  \dt.
\]
Using the Lebesgue differential theorem, we have $|t-\tau|^{-1}\int_{\tau}^{t}
w'(s)ds \underset{\tau \rightarrow t-}{\longrightarrow} w'(t)$ for a.a. $t$ and thus $\lim\limits_{\tau \rightarrow t^{-}} |t-\tau |^{-\alpha}
 |w(t)-w(\tau)|^{2}= \lim\limits_{\tau \rightarrow t^{-}} |t-\tau |^{2-\alpha}
\Big| |t-\tau|^{-1} \int_{\tau}^{t} w'(s) ds \Big|^{2}=0$.
Hence
\[
\frac{\alpha}{2} \jgja \int_{0}^{T} \izt  \taj   |w(t) -w(\tau)|^{2}   \dt  dt
+  \jd \jgja \int_{0}^{T}   t^{-\alpha}  |w(t) -w(0)|^{2}     dt
\]
\[
+ \frac{w(0)}{\gja} \inT t^{-\alpha} w(t) dt
+ \jd \frac{1}{\Gamma(2-\alpha)} \int_{0}^{T} (T-\tau)^{1-\alpha}
\left( |w(\tau)|^{2}  \right)_{\tau}  \dt
\]
\[
= \frac{\alpha}{2} \jgja \int_{0}^{T} \izt  \taj   |w(t) -w(\tau)|^{2}   \dt dt
+ \jd \jgja \int_{0}^{T}   t^{-\alpha}  |w(t) -w(0)|^{2}     dt
\]
\[
+ \frac{w(0)}{\gja} \inT t^{-\alpha} w(t) dt+\jd \frac{1}{\Gamma(1-\alpha)}
\int_{0}^{T} (T-\tau)^{-\alpha} |w(\tau)|^{2}  \dt
- \jd \frac{1}{\Gamma(2-\alpha)} T^{1-\alpha} |w(0)|^{2}
\]
\[
= \frac{\alpha}{2} \jgja \int_{0}^{T} \izt  \taj   |w(t) -w(\tau)|^{2}  \dt dt
+ \jd \frac{1}{\Gamma(1-\alpha)} \int_{0}^{T} (T-\tau)^{-\alpha} |w(\tau)|^{2}
\dt
\]
\[
+ \jd \frac{1}{\Gamma(1-\alpha)} \int_{0}^{T} \tau^{-\alpha} |w(\tau)|^{2}
\dt,
\]
and the proof is finished.
\end{proof}

\begin{coro}
Assume that  $t\in (0,T)$ and $\alpha \in (0,1)$.
If $w\in L^{2}(0,T)$ and $\ija w \in H^{1}(0,T)$, then (\ref{n11}) holds for
$w$.

\label{lowestiYa}
\end{coro}

\begin{proof}
According to \cite{Ya}, there exists a sequence $\{ w_{n} \}
\subset C^{1}[0,T]$ such that $w_{n}(0)=0$ and $w_{n}\rightarrow w$,
\hd $\ddt \ija w_{n} \rightarrow \ddt \ija w$ in $L^{2}(0,T)$.
Then applying (\ref{n11}) with $w_{n}$ and next taking the limit, we obtain
(\ref{n11}).
\end{proof}

\kr{
\no We recall some results from \cite{GY} and \cite{Ya}. We denote
\[
\hzj=\{ u \in H^{1}(0,T):\hd  u(0)=0  \}, \hd X_{\al}= \overline{  \spann{ \{ h_{n} \} }  }^{X_{\al}},
\]
where $h_{n}(t)=\sqrt{\frac{2}{T}} \sin(\frac{t}{\lambda_{n}})$, \hd $\la_{n}=\frac{T}{\pi(n+\frac{1}{2})}$, \hd $n=0,1, \dots$ and $X_{\al}$ is Hilbert space with the following inner product
\[
(u,v)_{X_{\al}}= \sum_{n=0}^{\infty} \la_{n}^{-2\al}(u,h_{n})_{L^{2}(0,T)} \overline{(v,h_{n})_{L^{2}(0,T)}}.
\]
By lemma~8 in \cite{GY} we have $X_{\al}=\hzaT$, where
\[
\hzaT=\left\{
\begin{array}{lll}
H^{\al}(0,T)  & \m{ for } & \al\in (0,\frac{1}{2}), \\
\{ u \in H^{\frac{1}{2}}(0,T): \hd \int_{0}^{T}\frac{|u(t)|^{2}}{t}dt<\infty \} & \m{ for } & \al=\frac{1}{2} \\
\{ u \in H^{\al}(0,T): \hd u(0)=0\} & \m{ for } & \al\in (\frac{1}{2},1),
\end{array}
\right.
\]
and for $\al\not = \frac{1}{2}$ we have $\| u \|_{\hzaT}=\| u\|_{H^{\al}(0,T)}$, but
\[
\| u \|_{{}_{0}H^{\frac{1}{2}}(0,T)}= \left( \| u \|^{2}_{H^{\frac{1}{2} }(0,T)}+\int_{0}^{T}\frac{|u(t)|^{2}}{t}dt \right)^{\frac{1}{2}}.
\]
From \cite{GY} and \cite{Ya} we deduce that for $\al\in [0,1]$ the operator $\ia: L^{2}(0,T)\longrightarrow \hzaT$ is isomorphism and the following inequalities
\eqq{e^{-\pi\sqrt{\al(1-\al)}  } \| u \|_{\hzaT} \leq \| \ra u \|_{L^{2}(0,T)}\leq e^{\pi\sqrt{\al(1-\al)}  } \| u \|_{\hzaT} \hd \m{ for } u \in \hzaT,}{YGa}
\eqq{e^{-\pi\sqrt{\al(1-\al)}  } \| \ia f  \|_{\hzaT} \leq \|  u \|_{L^{2}(0,T)}\leq e^{\pi\sqrt{\al(1-\al)}  } \| \ia f  \|_{\hzaT} \hd \m{ for } f  \in L^{2}(0,T),}{YGb}
holds. The above estimates are a consequence of Heinz-Kato theorem (see theorem 2.3.4 in \cite{Tanabe}).

For measurable $f$ defined on $(0,T)$ we set
\[
\tilde{f}(t)= \left\{ \begin{array}{cll}
f(t) & \m{ for } & t\in (0,T) \\
-f(-t) & \m{ for } & t\in (-T,0) \\
0 & \m{elsewhere.}\\
\end{array}
 \right.
\]
We define
\[
\Pi_{n}f(t)=\eta_{\frac{1}{n}}*\tilde{f}(t), \hd n=1,2, \dots,
\]
where $\eta_{\ep}$ is mollifier, i.e. $\eta_{\ep}\geq 0$, $\int_{\rr} \eta_{\ep}=1$, $\eta_{\ep}\in C^{\infty}_{0}(-\frac{\ep}{T}, \frac{\ep}{T})$ and we assume that in addition $\eta_{\ep}(t)=\eta_{\ep}(-t)$.

\begin{prop}
For each $\beta \in (0,1)$ and $n\in \N$ the following inequality holds
\eqq{\| \partial^{\beta} \Pi_{n} f \|_{L^{2}(0,T)}   \leq 2 e^{2\pi\sqrt{\beta(1-\beta)}  } \| \partial^{\beta}  f \|_{L^{2}(0,T)},   \hd \m{ if } \hd I^{1-\beta} f \in \hzjT.}{nova}
\label{estiapprox}
\end{prop}

\begin{proof}
By direct calculations we have
\[
\| \Pi_{n} f \|_{\hzjT} \leq 2 \| f \|_{\hzjT} \hd \m{ for } \hd f\in \hzjT,
\]
and
\[
\| \Pi_{n} f \|_{L^{2}(0,T)} \leq 2 \| f \|_{L^{2}(0,T)} \hd \m{ for } \hd f\in L^{2}(0,T).
\]
Thus by interpolation argument (see theorem~5.1 and remark 11.5 in \cite{Lions}) we have
\[
\| \Pi_{n} f \|_{{}_{0}H^{\beta}(0,T)} \leq 2 \| f \|_{{}_{0}H^{\beta}(0,T)} \hd \m{ for } \hd f\in {}_{0}H^{\beta}(0,T).
\]
Applying (\ref{YGa}) we obtain (\ref{nova}).
\end{proof}

\begin{prop}
Assume that $\beta\in [0,1)$ and $f\in L^{2}(0,T)$ satisfies  \m{$I^{1-\beta} f \in \hzj(0,T)$.} Then
\eqq{\partial^{\beta}\Pi_{n}f\longrightarrow \partial^{\beta} f \m{ \hd in \hd } L^{2}(0,T). }{popnowa}
Furthermore, this convergence is uniform with respect $\beta\in [0,\delta]$ for any $\delta\in (0,1)$.
\end{prop}
\begin{proof}
According to \cite{Ya}, the set ${}_{0}C^{1}([0,T])=\{ u\in C^{1}([0,T]): \hd u(0)=0\}$ is dense in ${}_{0}H^{\beta}(0,T)$. We fix $\ep>0$ and then from (\ref{YGa}) we deduce that there exists $\hat{f}\in {}_{0}C^{1}([0,T])$ such that
\[
\| \partial^{\beta} f - \partial^{\beta} \hat{f} \|_{L^{2}(0,T)}\leq \frac{\ep}{13}.
\]
Then, using (\ref{nova}) we have
\[
\| \rb \Pi_{n} f - \rb f\|_{L^{2}(0,T)}
\]
\[
\leq\| \rb \Pi_{n} f - \rb \Pi_{n}\hat{f} \|_{L^{2}(0,T)}+\| \rb \Pi_{n} \hat{f} - \rb \hat{f}\|_{L^{2}(0,T)}+\| \rb \hat{f} - \rb f\|_{L^{2}(0,T)}
\]
\[
\leq (2e^{2\pi \sqrt{\beta(1-\beta)}} +1)\frac{\ep}{13}+\| \rb \Pi_{n} \hat{f} - \rb \hat{f}\|_{L^{2}(0,T)}.
\]
To estimate the last term we write
\[
\| \rb \Pi_{n} \hat{f} - \rb \hat{f}\|_{L^{2}(0,T)} \leq e^{\pi \sqrt{\beta(1-\beta)}} \|  \Pi_{n} \hat{f} -  \hat{f}\|_{{}_{0}H^{\beta}(0,T)}
\]
\[
\leq e^{\pi \sqrt{\beta(1-\beta)}} [1+(1-\beta)^{-1}c(T)]^{1/2}\|  \Pi_{n} \hat{f} -  \hat{f}\|_{{}_{0}H^{1}(0,T)},
\]
where in the last inequality we applied the continuity of Hardy-Litlewood maximal operator in $L^{2}$. The last expression is estimated by $\ep/13$, provided $n$ is large enough and the estimate is uniform with respect to $\beta\in [0,\delta]$ for any $\delta \in (0,1)$.

\end{proof}

\no The above results can be extended to the case of vector value functions (see remark 11.5 in \cite{Lions}) and we have

\begin{prop}
If $H$ is a Hilbert space, then for each $\beta \in (0,1)$ and $n\in \N$ the following inequality holds
\eqq{\| \partial^{\beta} \Pi_{n} f \|_{L^{2}(0,T;H)}   \leq 2 e^{2\pi\sqrt{\beta(1-\beta)}  } \| \partial^{\beta}  f \|_{L^{2}(0,T;H)},   \hd \m{ if } \hd I^{1-\beta} f \in \hzj(0,T;H).
}{novaH}
Futhermore,
\eqq{\partial^{\beta}\Pi_{n}f\longrightarrow \partial^{\beta} f \m{ \hd in \hd } L^{2}(0,T;H) }{popnowab}
and this convergence is uniform with respect $\beta\in [0,\delta]$ for any $\delta\in (0,1)$.

\label{estiapproxH}
\end{prop}

}

\subsection*{Acknowledgment}

The research leading to these results has been supported by
the People Programme (Marie Curie Actions) of the European
Union's Seventh Framework Programme FP7/2007-2013/ under REA grant
agreement no 319012 and the Funds for
International Co-operation under Polish Ministry of Science and Higher
Education grant agreement no 2853/7.PR/2013/2.
Both authors are partially supported by Grants-in-Aid for Scientific
Research (S) 15H05740 and (S) 26220702, Japan Society for the
Promotion of Science.

\end{document}